\documentclass[12pt,reqno]{amsart}
\usepackage[top=1in,bottom=1in,left=1.25in,right=1.25in]{geometry}
\usepackage{amsmath,amssymb}
\usepackage{color}
\usepackage{amsfonts}
\usepackage[foot]{amsaddr}
\usepackage{mathtools}
\usepackage{setspace}
\usepackage[toc,page]{appendix}

\usepackage{hyperref}
\usepackage{cleveref}
\usepackage{enumitem}

\usepackage{amscd}
\usepackage{cjhebrew}
\usepackage[T1]{fontenc} 
\usepackage[utf8]{inputenc}
\usepackage{verbatim}
\usepackage{stmaryrd}
\usepackage{bbm}
\usepackage{mathrsfs}
\usepackage{todonotes}
\usepackage{thmtools}

\usepackage{enumerate}

\newtheorem{theorem}{Theorem}[section]
\newtheorem{proposition}[theorem]{Proposition}

\newtheorem{assumption}[theorem]{Assumption}
\newtheorem{lemma}[theorem]{Lemma}

\theoremstyle{definition}
\newtheorem{definition}[theorem]{Definition}

\theoremstyle{remark}
\newtheorem{remark}[theorem]{Remark}

\numberwithin{equation}{section}

\renewcommand{\Re}{\operatorname{Re}}

\newcommand{\la}{\lambda}

\def\Z{{\mathbb Z}}

\def\R{{\mathbb R}}

\newcommand{\norm}[1]{\left\|#1\right\|}

\newcommand{\T}{\mathbb{T}}

\renewcommand{\emptyset}{\varnothing}

\newcommand{\eps}{\varepsilon}

\newcommand{\one}{\mathbf{1}}

\allowdisplaybreaks

\title[]{On the effect of randomization on supercritical heat equations}

\author{Eliseo Luongo}
\address{Fakult\"at f\"ur Mathematik, Universit\"at Bielefeld, D-33501 Bielefeld, Germany}
\email{\href{mailto:eluongo at math.uni-bielefeld.de}{eluongo at math.uni-bielefeld.de}}

\date\today

\usepackage{kantlipsum}

\usepackage{fullpage}
\allowdisplaybreaks
\begin{document}

	\begin{abstract}
		Recently, in \cite{glogic2025non}, it has been shown that the focusing power nonlinearity heat equation
\begin{equation}\label{Eq:Heat_abstract}\tag{NLH}
			\partial_t u -\Delta u =  |u|^{p-1}u, \quad p>1,
		\end{equation}
		in dimensions $d \geq 3$ has non-unique local solutions in $L^q(\mathbb{R}^d)$ for $q < d(p-1)/2$ provided that $p < p_{JL}$, where $p_{JL}$ denotes the Joseph-Lundgren exponent. In this paper we investigate the effect of different randomizations on the well-posedness of the equation. First we show that adding a forcing term white in time and colored in space in \eqref{Eq:Heat_abstract} is not sufficient to improve the solution theory: namely, we prove non-uniqueness for local-in-time mild solutions of \eqref{Eq:Heat_abstract} with additive noise. Second, we discuss how randomizing the initial conditions of \eqref{Eq:Heat_abstract} affects its well-posedness.
	\end{abstract}
	
	\maketitle
	\section{Introduction}
	
	\noindent Consider the Cauchy problem for the focusing power nonlinearity heat equation
	\begin{equation}\label{Eq:Heat}
		\begin{cases}
			\partial_t u -\Delta u =  |u|^{p-1}u, \\
			u(0,\cdot)=u_0,
		\end{cases}
	\end{equation}
	where $u=u(t,x)\in \R$, $(t,x) \in [0,\infty) \times \R^d$, $d \geq 3$ and $p>1+\frac{2}{d}$.

	\begin{definition}\label{def_mild}
		Let $1 \leq q < \infty$, $u_0 \in L^q(\R^d)$ and $T>0$. By a \emph{mild $L^q$-solution} to \eqref{Eq:Heat} on the time interval $[0,T)$ 
		we call a function 
		\begin{equation*}
			u \in C([0,T),L^q(\R^d)) \cap L^p_{{loc}}((0,T)\times \R^d)
		\end{equation*}
		that is a distributional solution to \eqref{Eq:Heat} on $(0,T)\times \R^d$ and for which $u(0,\cdot)=u_0$.
	\end{definition}
    In the context of well-posedness theory for \eqref{Eq:Heat}, an important role is played by the exponent
    \begin{equation}\label{Def:q_c}
		q_c:= \frac{d(p-1)}{2}
	\end{equation}
    as deeply discussed in \cite{QuiSou19} and recalled in the following.\\
    It is known from the work of Weissler \cite{Wei80} that if $q > q_c$ or $q = q_c > 1$, then for arbitrary $u_0 \in L^q(\R^d)$ there exists $T>0$ and a mild $L^q$-solution to \eqref{Eq:Heat} on $[0,T)$. Furthermore, uniqueness in the space $C([0,T),L^q(\R^d))$ holds as soon as $q>q_c,\ q\geq p$ or $q\geq q_c,\ q>p$, see \cite{Wei80, BreCaz96}. However, neither of the techniques from the aforementioned two papers applies to $q=q_c=p$. In fact, uniqueness was later shown to fail in this case by Terraneo \cite{Ter02}. When $q < q_c$ it is not, in general, possible to associate with every $u_0 \in L^q(\R^d)$ a mild $L^q$-solution to \eqref{Eq:Heat}. Indeed, we refer to \cite{Wei80} for a non-existence result under the assumption of non-negativity of local solutions. In contrast, it is known that in certain cases there are initial data that lead to multiple mild solutions. The earliest result goes back to \cite{HarWei82}, who showed non-uniqueness for all $1 \leq q<q_c$ when the power $p$ is sufficiently small, belonging to the range
    \begin{equation}\label{Eq:Haraux_range}
		1 + \frac{2}{d} < p < 1+\frac{4}{d-2}.
	\end{equation}
    The non-uniqueness of \cite{HarWei82} follows by the existence of a non-trivial rapidly decaying expanding self-similar solution for \eqref{Eq:Heat}, which thereby arises from zero initial datum. Expanding self-similar profiles with rapid decay exist, however, only in the energy-subcritical range \eqref{Eq:Haraux_range}. Moreover, the construction of \cite{HarWei82} is very special since it allows one to prove only non-uniqueness arising from the trivial initial condition. Motivated by these reasons, recently in \cite{glogic2025non} we showed that non-uniqueness (from non-zero initial data) holds for the following range of powers
	\begin{equation}\label{Eq:Range}
		1+\frac{2}{d} < p < p_{JL},
	\end{equation}
	where $p_{JL}$ stands for the so-called Joseph-Lundgren exponent
	\begin{equation*}
		p_{JL}:=
		\begin{cases}
			\infty & \text{if }~3 \leq d \leq 10,\\
			1+ \displaystyle{\frac{4}{d-4 - 2 \sqrt{d-1}}} & \text{if }~d \geq 11. 
		\end{cases}
	\end{equation*}
	More precisely, we established the following result.
	\begin{theorem}{\cite[Theorem 1.2, Remark 1.5]{glogic2025non}}\label{Thm:main_deterministic}
		Assume $d \geq 3$ and let $p$ satisfy \eqref{Eq:Range}. Then for any $1 \leq q < q_c$
		there exists a non-trivial radial initial datum $u_0 \in L^q(\R^d)$ and $r>q_c$ sufficiently large depending on $p,d,q $ such that, for each radial initial datum $v_0 \in L^q(\R^d)\cap L^{r}(\R^d)$, there is a time $T>0$ for which there are two different mild $L^q$-solutions to \eqref{Eq:Heat} on $[0,T)$ starting from $u_0+v_0$.
	\end{theorem}
    The non-uniqueness of \autoref{Thm:main_deterministic} follows by the adaptation of the method introduced in \cite{JiaSve15} to show non-uniqueness of Leray-Hopf solutions for the unforced incompressible 3d Navier-Stokes equations. This approach is based on the assumption that there exists a forward self-similar solution that is linearly unstable in similarity variables. However, while for the 3d Navier-Stokes equation only numerical evidence of this holds, see \cite{GuiSve23}, we showed that this assumption is true for the equation \eqref{Eq:Heat} for the range of powers $p$ given by \eqref{Eq:Range}. We refer to \autoref{thm_spectrum} for the rigorous statement of the above claim.
    The initial profile $u_0$ constructed in \cite{glogic2025non} is singular, satisfying
\begin{equation}\label{form_IC}
			u_0(x) = \frac{C}{|x|^{\frac{2}{p-1}}} \quad \text{for some} \quad C>0,
		\end{equation}
    near $0$, it is uniformly bounded elsewhere and $C^2(\R^d\setminus\{0\})$. \autoref{Thm:main_deterministic} also says that the non-uniqueness from the singular initial data $u_0$ is robust enough that the non-uniqueness mechanism is preserved in the case of \emph{regular} perturbations, $v_0$, of the singular initial condition, $u_0$.  In this paper, we go further in this direction to analyze the robustness of the non-uniqueness mechanism described above and to understand whether possible randomizations of \eqref{Eq:Heat} may improve its well-posedness theory. Indeed, if the pathological behavior of \eqref{Eq:Heat} occurs only along exceptional trajectories and is not generic, a small random perturbation can move the system away from them, thereby restoring uniqueness. The above heuristic had several applications in the last 40 years both in finite and infinite dimensional systems: we refer to the classical works \cite{zvonkin1974transformation, veretennikov1981strong} and the more recent results \cite{flandoli2010well,DP_regular_1, DP_regular_2,burq2008random, Priola, bertacco2023weak, beck2019stochastic, Saleh, coghi2023existence, herr2023three,agresti2024global} including the lecture notes \cite{flandoli2011random} for a review of this topic. This approach has as a final aim the possibility, by letting the intensity of the noise approach $0$, to select some special solutions among the non-unique ones of the deterministic system. Even if this program seems very appealing, it is really hard to see it in practice. We are aware of very few examples where it has been completely settled \cite{bafico1982small, attanasio2009zero,  delarue2014transition, delarue2014noise, delarue2019zero,crippa2025zero}.\\
In this paper we follow two different approaches in order to randomize \eqref{Eq:Heat} with the aim to obtain a regularization by noise phenomenon. In the first case, in the same spirit of \cite{zvonkin1974transformation, veretennikov1981strong, DP_regular_1,DP_regular_2, Priola} we study \eqref{Eq:Heat} with additive noise, proving that \eqref{Eq:Heat} perturbed by additive noise, white in time and colored in space, still suffers from the same non-uniqueness mechanism of \autoref{Thm:main_deterministic}. We refer to \autoref{intro_additive_noise} for the rigorous presentation of the main result in this framework. Secondly, we decided to study \eqref{Eq:Heat} without adding any additional random force but providing a randomization of its initial conditions in the spirit of \cite{burq2008random, burq2008randomII, zhang2011random, pocovnicu2017almost}. In this case we are able to
construct a unique local mild solution for a large set of initial data, we refer to \autoref{intro_random_initial_cond} for the rigorous presentation of the main result in this framework.

    \subsection{The case of additive noise}\label{intro_additive_noise}
As discussed above, our first aim is to study the failure of local well-posedness in $L^q(\R^d),\ q<q_c$ for the stochastic evolution equation on $\mathbb{R}^d$
\begin{align}\label{stochastic_heat}
    \begin{cases}
        d u&=(\Delta u+\lvert u\rvert^{p-1}u)dt +dW_t\\
        u(0)&=u_0.
    \end{cases}
\end{align}    
Above $W_t$ is an infinite dimensional Brownian motion on a filtered probability space  $(\Omega,\mathcal{F},$ $(\mathcal{F}_t)_{t\geq 0},\mathbb{P})$. Similarly to the deterministic case, we are interested in studying mild solutions of \eqref{stochastic_heat}. For the sake of presentation, we postpone this definition, pretty analogous to \autoref{def_mild}, to \autoref{sec_preliminaries} below. 
In \autoref{proof_non_unique} we will be able to prove the following.
\begin{theorem}\label{Thm:main_additive_noise}
  Assume $d \geq 3$, \autoref{HP_noise_1} and let $p$ satisfy \eqref{Eq:Range}. Then for any $1 \leq q < q_c$
		there exists a non-trivial initial datum $u_0 \in L^q(\R^d)$ and a stopping time $\mathcal{T}>0$ for which there are two different mild $L^q$-solutions to \eqref{stochastic_heat} on $[0,\mathcal{T}]$.
\end{theorem}
\begin{remark}\label{rmk_more_general_IC}
The $u_0$ we will construct in \autoref{Thm:main_additive_noise} is of the same form as \eqref{form_IC}. Similarly to \autoref{Thm:main_deterministic}, as will be apparent from the proof, the non-uniqueness mechanism we exhibit persists under perturbations of $u_0$ that are radial and in $L^q \cap L^r$ for certain $r>q_c$. Note, however, that such perturbations do not remove the singular behavior near zero.  
\end{remark}
The non-uniqueness of \autoref{Thm:main_additive_noise} is again based on the existence of a forward self-similar solution that is linearly unstable in similarity variables for the deterministic unforced equation in the spirit of \cite{JiaSve15}. Therefore, \autoref{Thm:main_additive_noise} shows that the non-uniqueness mechanism of \autoref{Thm:main_deterministic} persists even in the presence of stochastic forcing, as long as the noise is sufficiently regular in space, this being the only assumption imposed in \autoref{HP_noise_1}. In particular, the addition of a stochastic forcing, does not spoil any of the parameters range available in our previous deterministic result. Even if this goes in the opposite direction to the papers quoted above, it is not completely surprising. Indeed, in the last few years several deterministic non-uniqueness results, in particular those based on convex integration techniques for equations in fluid mechanics \cite{buckmaster2015anomalous,isett2018proof, buckmaster2019nonuniqueness}, have been extended to the stochastic framework, allowing either for the additive structure of the noise as in \eqref{stochastic_heat}, see \cite{hofmanova2022ill,hofmanova2023nonuniqueness,hofmanova2024global,lu2025proof}, or  multiplicative ones, see \cite{berkemeier20233d, hofmanova2024global,koley2025non,yamazaki2023non}. In this sense, \cite{hofmanova2024non,brue2023non} are the most similar results compared to \autoref{Thm:main_additive_noise}. They are based on the non-uniqueness for Leray-Hopf solutions for the 3d Navier-Stokes equations \cite{AlbBruCol22}, proving that the same holds for forced 3d stochastic Navier-Stokes equation with additive or linear multiplicative noise. However, some differences arise. Similarly to \cite{AlbBruCol22}, either \cite{hofmanova2024non} and \cite{brue2023non} require adding a further external forcing to produce the non-uniqueness. Moreover, the additional forcings differ from that of \cite{AlbBruCol22}, being in both cases proper stochastic processes. This is due to the fact that, while 
here we do not force the unstable self-similar profile constructed in \cite{glogic2025non} to be a solution of \eqref{stochastic_heat}, see \autoref{sec:localization}, this is done in \cite{brue2023non, hofmanova2024non}, leading to a different forcing for each realization of the noise. Ultimately, also this difference seems caused by the fact that we were able in \cite{glogic2025non} to verify the assumptions of the program proposed by Jia and {\v{S}}ver{\'a}k in \cite{JiaSve15} without the necessity to add an additional external forcing.\\

A natural conjecture, linked to the well-known motto \emph{the rougher the noise, the better the regularization}, is that the failure of the regularization by noise mechanism and therefore the persistence of the non-uniqueness for \eqref{stochastic_heat} is hidden in \autoref{HP_noise_1}. It could be the case that a rougher noise does not allow us to study \eqref{nonlinear_pde} and ultimately to restore uniqueness for \eqref{stochastic_heat}.

    \subsection{The case of randomization in the initial conditions}\label{intro_random_initial_cond}
    The non-uniqueness in \eqref{Eq:Heat} of \autoref{Thm:main_deterministic} is \emph{typical} of nonlinear PDEs settled in supercritical spaces for which several ill-posedness phenomena, like non-uniqueness or instantaneous lack of regularity, are possible \cite{ChrColTao03,lebeau2005perte}. However, the functions for which one can rigorously prove such a pathological behavior are, usually, highly non-generic and one would expect that the ill-posedness is indeed a non generic phenomenon. Therefore, a Gaussian measure on the space of initial conditions should be able to recognize if there is a lack of the known ill-posedness results, telling us if they are \emph{generic} in terms of initial conditions or not. Starting from the seminal works \cite{burq2008random, burq2008randomII} the above heuristic has been applied in several contexts, we refer for example to \cite{nahmod2013almost,zhang2011random, pocovnicu2017almost,luhrmann2014random} and the references therein for several applications of this principle. In the context of \autoref{Thm:main_deterministic}, given $u_0$ as the one in the statement, one can only show that for each $\epsilon>0$ and radial initial condition ${u}_{0,*}\in L^q(\R^d)$, there exists $v_0$ radial and smooth such that \begin{align*}
        \norm{{u}_{0,*}-u_0-v_0}_{L^q}\leq \epsilon
    \end{align*} 
    and \eqref{Eq:Heat} with initial condition $u_0+v_0$ has non-unique solutions. Motivated by the above considerations, we are interested in studying \eqref{Eq:Heat} with random initial conditions.\\
    Similarly to \cite{luhrmann2014random,pocovnicu2017almost} we randomize functions directly on Euclidean space via a unit-scale decomposition in frequency space.
    More precisely, let $\varphi\in C^{\infty}_c(\R^d)$ be a real-valued function such that $0 \leq \varphi\leq 1,\ \varphi(\xi)=\varphi(-\xi)$ and
\begin{align*}
    \varphi(\xi)=\begin{cases} 
        1  \quad &\text{for }  \xi\in [-1/2,1/2]^d\\
        0 \quad &\text{for }  \xi\notin [-1,1]^d.
        \end{cases}
\end{align*}
For every $k\in \Z^d$ set $\varphi_k(\xi)=\varphi(\xi-k)$ and define
\begin{align}\label{def_multipliers}
    \psi_k(\xi)=\frac{\varphi_k(\xi)}{\sum_{l\in \Z^d}\varphi_l(\xi)}.
\end{align}
Then $\psi_k$ is a smooth function with support contained in $k+[-1,1]^d$ and $\psi_k(\xi)=\psi_{-k}(-\xi),$ $ \sum_{k\in \Z^d}\psi_k(\xi)\equiv 1\ \forall\xi \in \R^d.$\\
Secondly, let $\mathcal{I}\subseteq \Z^d$ be such that $\mathbb{Z}^d=\mathcal{I}\cup (-\mathcal{I})\cup \{0\},\ \mathcal{I}\cap  (-\mathcal{I})=\emptyset $ and 
$\{h^1_k,h^2_k\}_{k\in \{0\}\cup \mathcal{I}}$ be i.i.d. standard real Gaussian random variables $\mathcal{F}_0$-measurable so that $h_k=h^1_k+ih^2_k$ is a complex Gaussian random variable $\mathcal{F}_0$-measurable. Set also $h_k=\overline{h_{-k}}$ for $k\in -\mathcal{I}.$ Lastly, for $f\in \mathscr{S}'(\R^d)$ let us define
\begin{align}\label{fourier_k_multiplier}
    P_k f:=\mathcal{F}^{-1}\left(\psi_k(\xi)\hat{f}(\xi)\right)\in \mathscr{S}'(\R^d).
\end{align}
Above, we denoted by $\hat{f}$ the Fourier transform of $f$, which is well-defined since $f$ is a tempered distribution and by $\mathcal{F}^{-1}$ the inverse Fourier transform.  \\
Having this notation in mind we define the randomization of $f$, denoted by $f^{\omega}$ 
\begin{align}\label{randomization}
    f^{\omega}=\sum_{k\in \Z^d}h_k(\omega)P_k f.
\end{align}
The definition of randomization of $f$ above is strongly linked to the notion of modulation spaces, $M^{q_1,q_2}(\R^d)\ 1\leq q_1,q_2\leq +\infty$, for which we refer to \cite{feichtinger1983modulation} for a review and standard results. Indeed, not surprisingly, if $f\in M^{q,q}(\R^d),\ 1\leq q\leq 2$ then $f^{\omega}$ above is well-defined and satisfies suitable integrability properties. We refer to \autoref{reg_initial_condition} below for details. We just recall here, for the sake of presentation of \autoref{Fixed_point_random_initial_cond}, that $M^{2,2}(\R^d)=L^2(\R^d)$ and $M^{q,q}(\R^d)\hookrightarrow L^q(\R^d)$ for $1\leq q<2.$\\
Now we are ready to state our main result for \eqref{Eq:Heat} with randomized initial conditions. Denoting by $e^{t\Delta}$ the heat semigroup, in \autoref{Random_data_well_posed} we shall prove the following result.
\begin{theorem}\label{Fixed_point_random_initial_cond}
    Let $d\geq 3,\ p>\frac{1}{2}+\sqrt{\frac{1}{4}+\frac{4}{d}},\ q\in \left(2-\frac{8}{4+dp},q_c\right)$, $r=(p\vee p')q_c$,
    \begin{align*}
        u_0\in \begin{cases}
            M^{q,q}(\R^d) &\quad \text{if }q< 2;\\
            L^{2}(\R^d)\cap L^q(\R^d)&\quad \text{if }q\geq 2
        \end{cases}
    \end{align*}
     and $u_0^{\omega}$ the randomization of $u_0$ given by \eqref{randomization}. Then there is an event $\Sigma\subseteq \Omega$ with $\mathbb{P}(\Sigma)= 1$ and a stopping time $\mathcal{T}$, such that $\mathcal{T}>0\ \mathbb{P}-a.s.$ and for any $\omega\in \Sigma$, the problem \eqref{Eq:Heat} with initial condition $u_0^{\omega}$
has a unique mild solution $u$ for $0 \leq t \leq  \mathcal{T}(\omega)$, satisfying 
\begin{align*}
    u-e^{t\Delta}u_0^{\omega}\in C([0,\mathcal{T}(\omega)];L^{q}(\R^d)\cap L^r(\R^d)).
\end{align*}
Moreover, $u(t\wedge\mathcal{T})$ is an adapted process with values in $L^q(\R^d)$. Lastly, there exist constants $C^1_{d,p,r}, C^2_{d,p,q,r}>0$ depending on their subscripts and,
for every $0<T\leq 1$, an event $\Omega_T$ with the property
\begin{align*}
    \mathbb{P}(\Omega_T)\geq \begin{cases}
    1- C^1_{d,p,r} e^{-\frac{C^2_{d,p,q,r}}{T^{\frac{1}{p}\wedge\left(\frac{2}{p-1}-\frac{d}{r-p+1}\right)}\norm{u_0}_{M^{q,q}}^2}}  &\quad \text{if }q< 2,
    \\    1- C^1_{d,p,r} e^{-\frac{C^2_{d,p,r}}{T^{\frac{1}{p}}\norm{u_0}_{L^2}^2}} &\quad \text{if }q\geq 2,
    \end{cases}
\end{align*}
such that, for every $\omega\in \Omega_T$, it holds $\mathcal{T}(\omega)\geq T.$
\end{theorem}
Some remarks are in order.
\begin{remark}
    Our restriction on $q$, namely the fact that $q>2-\frac{8}{4+dp}$ and we cannot reach the full range of parameters $q\in [1,q_c)$ seems linked to similar issues appearing for randomization in Navier-Stokes equations, see for example \cite[Theorem 2.4]{nahmod2013almost}, \cite[Theorem 1.1]{deng2011random}, where no arbitrarily low Sobolev regularity can be reached. Indeed, even working in the $d$-dimensional torus, the space $H^{-\frac{d}{2}}(\T^d)$, having the same Sobolev index as $L^1(\T^d)$, is not reached by the quoted papers. Note in particular that the interval $\left(2-\frac{8}{4+dp},q_c\right)$ is not empty as soon as $p>\frac{1}{2}+\sqrt{\frac{1}{4}+\frac{4}{d}}$. This is the reason behind the stronger restriction on $p$ in the statement compared to the more standard $p>1+\frac{2}{d}$.
\end{remark}
\begin{remark}
    In the case of $q>2 $ we require also $u_0\in L^2(\R^d)$. This is a technical assumption that allows us to make sense to the series appearing in \eqref{randomization} and to define $u_0^{\omega}$ as a random variable with values in $L^2(\R^d)$. As \autoref{reg_initial_condition} below shows, $u_0^{\omega}\in L^{q}(\R^d)$ in this case without assuming $u_0\in L^q(\R^d)$. We preferred to state the result in this suboptimal way in order to be somehow coherent to our purpose of studying the Cauchy problem of \eqref{Eq:Heat} in $L^q(\R^d)$. Note that this assumption is satisfied by the singular profile $u_0$ in \autoref{Thm:main_deterministic} according to \eqref{form_IC}.
\end{remark}
\begin{remark}
    We have the embedding $W^{d\left(\frac{2}{q}-1\right),q}(\R^d)\hookrightarrow M^{q,q}(\R^d)$ for $q<2$, \cite{sugimoto2007dilation,toft2004continuity}. Moreover, the $u_0$ of \autoref{Thm:main_deterministic} belongs to $W^{\theta,q}(\R^d)$ for $\theta<\frac{d}{q}-\frac{2}{p-1}$ according to \eqref{form_IC}. In particular it is a proper initial condition for \autoref{Fixed_point_random_initial_cond} if \begin{align*}
        p\geq 3 \quad\text{or }\frac{(d-2)(p-1)}{(3-p)}>q.
    \end{align*}
    This interval is included in the whole $\left(2-\frac{8}{d+dp},q_c\right)$ as soon as $p\geq 1+\frac{4}{d}$. The latter is slightly more restrictive than $p>\frac{1}{2}+\sqrt{\frac{1}{4}+\frac{4}{d}}$ assumed in \autoref{Fixed_point_random_initial_cond}.
\end{remark}
The idea behind the proof, similarly to \cite{burq2008random}, is that, besides the fact that the randomization does not allow to gain differentiability, see \cite[Appendix B]{burq2008random}, the $L^r$ properties are (almost surely) much better than expected. This improvement allows the use of a fixed point method after having singled out the linear evolution. This is easily transparent in case of $q\geq 2$ in view of \autoref{reg_initial_condition} and \autoref{further_integrability_rmk} below, slightly less in the case of $q<2$. Indeed, to the best of our knowledge, all previous quoted results in the same spirit of \autoref{Fixed_point_random_initial_cond} were established by exploiting the effect of the randomization in a Hilbert space framework. In case $q<2$, however, due to embedding $M^{q,q}(\R^d)\hookrightarrow L^q(\R^d)\hookrightarrow H^{-\frac{d(2-q)}{2q}}(\R^d)$ \emph{cf.} \cite{sugimoto2007dilation,toft2004continuity}, we can still recover a Hilbert space framework where it is possible to exploit the regularizing properties of the randomization \eqref{randomization} and, at the same time, work with $L^q$ mild solutions for \eqref{Eq:Heat}.\\

We regard \autoref{Fixed_point_random_initial_cond} as a complementary result to \autoref{Thm:main_additive_noise}. According to \autoref{Fixed_point_random_initial_cond}, the non-uniqueness scenario described in \autoref{Thm:main_deterministic} appears to be non-generic, and the introduction of suitably chosen stochastic forcings should be able to regularize the system. However, as \autoref{Thm:main_additive_noise} shows, such stochastic forcings cannot be chosen too naively. A more precise regularity threshold for the additive noise, or alternatively more refined noise structures tailored to the features of \eqref{Eq:Heat}, would be required. As already mentioned, the ultimate goal of this program would be to establish selection criteria for the non-unique solutions of \eqref{Eq:Heat}. We consider this to be a problem of very high interest. Indeed, although \eqref{Eq:Heat} is a simplified model for nonlinear parabolic systems, we believe that the techniques needed to address the questions above could significantly enhance our understanding of how noise interacts with nonlinear parabolic equations and, ultimately, contribute to a deeper understanding of the dynamics of more complex and physically relevant partial differential equations.
\subsection{Plan of the paper}    
The content of the paper is as follows. In \autoref{sec_preliminaries} we settle the notation that we follow throughout the paper, recalling, for the convenience of the reader, some results from our previous paper \cite{glogic2025non} and on modulation spaces. The proof of \autoref{Thm:main_additive_noise} is the content of \autoref{proof_non_unique}, while the proof of \autoref{Fixed_point_random_initial_cond} is the content of \autoref{Random_data_well_posed}.
\section{Preliminaries and Notation}\label{sec_preliminaries}
\subsection{Notation and Conventions}\label{FA_not}
Let us start by introducing an appropriate functional framework.
For $p \in [1,+\infty]$,
we denote by $L^{p}(\R^d)$ the space of Lebesgue functions, by $C_c^{\infty}(\R^d)$ the standard test space of smooth and compactly supported functions on $\R^d$, by $\mathscr{S}(\R^d)$ the space of Schwartz functions on $\R^d$ and by $\mathscr{S}'(\R^d)$ the dual of $\mathscr{S}(\R^d)$, namely the space of tempered distributions. Let us also settle some standard notation for Sobolev spaces. For $s\in \R,\ p\in [1,+\infty]$ we denote by $W^{s,p}(\R^d)$ the Sobolev–Slobodeckij space of differentiability $s$ and integrability $p$. In the case of $p=2$ we simply write $H^{s}(\R^d)$ in place of $W^{s,2}(\R^d)$. We refer to the classical reference \cite{triebel} for a detailed discussion and standard results on Sobolev spaces.

As already announced in \autoref{intro_random_initial_cond}, we also need to introduce modulation spaces on $\R^d$. Let  $P_k:\mathscr{S}'(\R^d)\rightarrow \mathscr{S}'(\R^d)$ be the linear operator introduced in \eqref{fourier_k_multiplier}. For $p,q\in [1,+\infty],\ s\in \R$ we define the modulation space $M^{p,q}_s(\R^d)$ as the space of tempered distributions $f\in \mathscr{S}'(\R^d)$ such that the norm
\begin{align}\label{modulation_norm}
\norm{f}_{M^{p,q}_s}=\left(\sum_{k\in \Z^d}\left(1+\lvert k\rvert^2\right)^{\frac{sq}{2}}\norm{P_k f}_{L^p}^q\right)^{1/q}
\end{align}
is finite, with the usual modification if $q=+\infty$. Modulation spaces are Banach spaces, moreover $\mathscr{S}(\R^d)$ is dense in $M^{p,q}_s$ whenever $p\vee q<+\infty.$ The definition is quite similar to that of Besov spaces, the main difference being that, in the latter, the frequency space is decomposed into dyadic annuli. When $s=0$, we simply write $M^{p,q}(\R^d)$ in place of $M^{p,q}_0(\R^d)$ in case of $s=0.$ We refer to \cite{feichtinger1983modulation} for a detailed discussion on modulation spaces and their standard properties including the embedding $M^{p,p}(\R^n)\hookrightarrow L^{p}(\R^n)$ for $p<2$, which we shall employ below.

Secondly, in order to deal with \autoref{Thm:main_additive_noise}, we need to introduce the framework of radial functions, also to accommodate the results from \cite{glogic2025non} recalled in \autoref{results_previous_paper}.
For $p \geq 1$, we define the space of radial Lebesgue functions
	\begin{equation*}
		L^p_{rad}(\R^d):= \{ f : [0,\infty) \rightarrow \mathbb{C}~ \vert~ f \text{ is measurable and } \| f \|_{L^p_{rad}(\R^d)} := \| f(|\cdot|) \|_{L^p(\R^d)} < \infty  \}.
	\end{equation*}
	We also need the spaces 
	\begin{align*}
		C_{c,rad}^\infty(\R^d) &:= \{ f : [0,\infty) \rightarrow \mathbb{C} ~ \vert ~ f(|\cdot|) \in C_c^{\infty}(\R^d)  \},\\
		\mathscr{S}_{rad}(\R^d)&:= \{ f : [0,\infty) \rightarrow \mathbb{C} ~ \vert ~ f(|\cdot|) \in \mathscr{S}(\R^d)  \}.   
	\end{align*}
	For convenience, we will often shortly write $L^p$ for $L^p (\R^d)$ and analogously for the other function spaces defined in this section. Now, for $1\leq \eta\leq \gamma$ we define the radial intersection Lebesgue space\footnote{We decided, similarly to \cite{glogic2025non}, to use the suggestive notation $L^{\eta,\gamma}$, hoping it will not cause confusion with the more standard usage in the context of Morrey or Lorentz spaces. }
	\begin{equation*}
		L^{\eta,\gamma}:=L^\eta_{{rad}} (\R^d) \cap L^\gamma_{{rad}} (\R^d), \quad \| \cdot \|_{L^{\eta,\gamma}} := \| \cdot \|_{L^\eta(\R^d)} + \mathbf{1}_{(0,\infty)}(\gamma-\eta)\| \cdot \|_{L^\gamma(\R^d)}.
	\end{equation*}
	Note that $L^{\eta,\eta}=L_{rad}^\eta$. Although we define spaces of radial functions on $\R^d$ via their radial profiles, for convenience we will at times interpret them as defined on $\R^d$ via the identification $f(x)=f(|x|)$ for $x \in \R^d$.

	Given a closed linear operator $(L,\mathcal{D}(L))$ on a Banach space $X$,  we denote by $\rho(L)$ the resolvent set of $L$, while $\sigma( L):= \mathbb{C} \setminus \rho( L)$ stands for the spectrum of $ L$.\\ Given a Hilbert space $U$, we denote by $\gamma(U,X)$ the space of $\gamma$-radonifying operators between $U$ and $X.$
    For estimates, we use the convenient asymptotic notation $a \lesssim b$ to say that there is some $C>0$ such that $a \leq Cb$. Sometimes, when it is obvious from the context, we omit explicitly mentioning the parameters on which the choice of the implied constant $C$ does not depend. To emphasize the dependence of $C$ on a parameter, say $p$, we will sometimes write $\lesssim_p$.
	
\subsection{Assumptions on the noise}\label{Noise_not}
Let us begin by introducing a complete filtered probability space $(\Omega,\mathcal{F},(\mathcal{F}_t)_{t\geq 0},\mathbb{P})$ with right continuous filtration such that $\mathcal{F}_0$ contains all $\mathbb{P}$-negligible sets. For the Brownian motion $W$ appearing in \eqref{stochastic_heat}, we assume the following:
\begin{assumption}\label{HP_noise_1}
		Let $U$ be a separable Hilbert space and $B_t$ be cylindrical Brownian motion on $U$,
        \begin{align*}
            W_t=JB_t,\ J\in \gamma(U;L^q_{rad}(\R^d))\cap \gamma(U;W^{s,q})\quad \text{for some }s=\left(\frac{d}{q}\left(1-\frac{2q}{p(p-1)(d+2q)}\right)-1\right)\vee 0 .
        \end{align*}
        
	\end{assumption}
	Under these assumptions on $W,$ according to \cite[Theorem 10.18]{van2008stochastic_lec}, there exists a unique mild (equivalently weak) solution to the stochastic heat equation on $\R^d$
	\begin{align}\label{Eq:Stoch_lin_heat}
		\begin{cases}
			dz&=\Delta zdt +dW_t\\
			z(0)&=0.
		\end{cases}
	\end{align}
	Namely there exists a unique $z\in C_{loc}([0,+\infty);L^q_{rad}(\R^d))$ solving the evolution problem \eqref{Eq:Stoch_lin_heat}.
	It is given by the mild formula \begin{align*}
	    z(t)=\int_0^t e^{(t-s)\Delta}dW_s,
	\end{align*}
    where we denote by $e^{t\Delta}$ the heat semigroup on $\R^d.$
Moreover, according to \cite[Chapter 10.3]{van2008stochastic_lec}, \cite[Theorem 1.1]{van2008stochastic}, \cite[Theorem 1.2]{Veraar}, due to the regularizing properties of the heat semigroup and the Sobolev embedding $W^{s+1,q}(\R^d)\hookrightarrow L^{\frac{p(p-1)(d+2q)}{2}}(\R^d)$ we also have
\begin{align}\label{eq:regularity_stoch_conv}
    z\in C_{loc}([0,+\infty);L^{q,\theta}) \quad \forall \theta<\frac{p(p-1)(d+2q)}{2}\quad\mathbb{P}-a.s.
\end{align}
Note that this also implies that $z\in L^p_{loc}((0,+\infty)\times \R^d)\quad\mathbb{P}-a.s.$

With these definitions at hand, similarly to \autoref{def_mild}, we can introduce the notion of mild solution for equation \eqref{stochastic_heat}.
\begin{definition}\label{stoc_mild}
		Let $u_0\in L^q(\R^d)$ and $\mathcal{T}$ a stopping time. A progressively measurable process $u:\Omega\times [0,+\infty)\rightarrow L^q(\R^d)$ is a \emph{mild solution} of \eqref{stochastic_heat} on $[0,\mathcal{T}]$ if $u$ belongs to $C([0, \mathcal{T}]; L^q(\R^d))\cap L^p_{loc}((0,\mathcal{T})\times \R^d)$ $\mathbb{P}$-a.s., $u(0)=u_0$ $\mathbb{P}$-a.s. and
        $v:=u-z$ is a distributional solution to
		\begin{align*}
        \partial_t v -\Delta v =  |v+z|^{p-1}(v+z).
		\end{align*}
        on $(0,\mathcal{T})\times \R^d $ $\mathbb{P}$-a.s.
\end{definition}

\subsection{Useful results from \cite{glogic2025non}}\label{results_previous_paper}
As already mentioned in \autoref{intro_additive_noise}, the non-uniqueness of \autoref{Thm:main_additive_noise} is a consequence of the existence of a forward self-similar solution that is linearly unstable in similarity variables for the deterministic unforced equation. \\
A radial forward self-similar solution for \eqref{Eq:Heat} corresponds to a self-similar profile 
\begin{align}\label{ss_profile}
    u(t,r)=\frac{1}{t^{\frac{1}{p-1}}}U\left(\frac{r}{\sqrt{t}}\right),
\end{align}
with $U$ satisfying the nonlinear ODE
\begin{align}\label{ODE}
    U'' + \left(\frac{d-1}{\rho}+\frac{\rho}{2}  \right)U' + \frac{1}{p-1} U + |U|^{p-1}U = 0.
\end{align}
According to \cite{HarWei82}, the above problem has a unique solution, up to fixing $U(0):=\alpha>0,\ U'(0)=0$. More precisely, the following holds.
	\begin{proposition}[\cite{HarWei82}]\label{Prop:Har-Wei}
		Let $d \geq 3$ and $p>1+\frac{2}{d}$. Then for every $\alpha > 0$ there exists a unique function $U \in C^2_{loc}([0,\infty))$ that satisfies \eqref{ODE} on $(0,\infty)$ classically and $U(0):=\alpha>0,\ U'(0)=0$. Moreover, $U$ is bounded and $\lim_{\rho \rightarrow \infty} \rho^{\frac{2}{p-1}}U(\rho)$  exists and is finite.
	\end{proposition}
    To indicate the dependence on $\alpha$, in what follows we denote the expanders from the theorem above by $U_\alpha$. Furthermore, we denote
	\begin{equation}\label{Def:ell}
		\ell(\alpha):=\lim_{\rho \rightarrow \infty} \rho^{\frac{2}{p-1}}U_\alpha(\rho).
	\end{equation}
    To analyze stability of the constructed expanders, it is customary to rewrite \eqref{Eq:Heat} in variables that are adapted to the self-similar nature of \eqref{ss_profile}, the so-called \emph{(radial) similarity variables}
    	\begin{equation*}
		\tau :=  \ln t, \quad \rho := \frac{r}{\sqrt{t}}.
	\end{equation*}
	By also scaling the dependent variable
	\begin{equation*}
		v(\tau,\rho):=t^\frac{1}{p-1}u(t,r),
	\end{equation*}
	from \eqref{Eq:Heat} we arrive at an evolution equation for  $v$
	\begin{equation}\label{Eq:Sim_var_v}
		\partial_\tau v = L_0 v + |v|^{p-1}v,
	\end{equation}
	where the linear operator $L_0$ is given by
	\begin{equation*}
		L_0  = \partial_\rho^2 +\frac{d-1}{\rho} \partial_\rho + \frac{1}{2}\rho \partial_\rho + \frac{1}{p-1}.
	\end{equation*}
	Note that, by definition, expander profiles $U_\alpha$ are now static solutions of \eqref{Eq:Sim_var_v}. What we did in \cite[Section 2]{glogic2025non} is the (non)linear stability analysis of $U_\alpha$. Namely we studied the flow of \eqref{Eq:Sim_var_v} near $U_\alpha$ in spaces $L^{\eta,\gamma}$. To this end, we set $v(\tau,\rho)= U_\alpha(\rho) + w(\tau,\rho)$ in \eqref{Eq:Sim_var_v}. This then leads to an evolution equation for the perturbation $w$
	\begin{equation}\label{Eq:Sim_var_w}
		\partial_\tau w = L_{\alpha}w + N(w),
	\end{equation}
	where
	\begin{equation*}
		L_{\alpha}=L_0+V_\alpha,\quad  V_\alpha= p  |U_\alpha|^{p-1},
	\end{equation*}
	and
	\begin{equation*}
		\quad N_\alpha(w) := n(U_\alpha+w)-n(U_\alpha)-p|U|^{p-1}w \quad \text{for} \quad  n(f)=|f|^{p-1}f. 
	\end{equation*}
    Stability of the forward self-similar solutions $U_{\alpha}$ in similarity variables corresponds exactly to spectral analysis for the operator $L_{\alpha}.$
The main findings of \cite[Section 2]{glogic2025non} can be summed up by the following theorem. 
\begin{theorem}[\cite{glogic2025non}]\label{thm_spectrum}
Let $d\geq 3,\ p>1+\frac{2}{d},\ 1\leq \eta\leq \gamma$. Then for each $\alpha>0$ the operator $L_\alpha:\mathcal{D}(L_\alpha) \subseteq L^{\eta,\gamma} \rightarrow L^{\eta,\gamma}$ generates a one-parameter strongly continuous semigroup $\left(S_{\alpha}(\tau)\right)_{\tau \geq 0}\subseteq \mathcal{L}(L^{\eta,\gamma})$. Assuming furthermore that
		\begin{equation}
			\quad  1 \leq \eta < \frac{d(p-1)}{2}, 
		\end{equation}
		then for the operator $L_\alpha : \mathcal{D}(L_\alpha) \subseteq L^{\eta,\gamma} \rightarrow L^{\eta,\gamma}$ the following statements hold.
		\begin{itemize}[leftmargin=8mm]
			\setlength{\itemsep}{2mm}
			\item[1.] The set
			\begin{equation}\label{Eq:unstab_spectr}
				\sigma(L_\alpha) \cap \big\{ \la \in \mathbb{C}~ \lvert~ \Re \la > \tfrac{1}{p-1}-\tfrac{d}{2 \eta} \big\}
			\end{equation}
			consists of finitely many real eigenvalues.
			\item [2.] If $p < p_{JL}$ then for every $\varepsilon>0$ there exists $\alpha>0$ such that $L_\alpha$ admits at least one positive eigenvalue, and furthermore all positive eigenvalues are smaller than $\varepsilon$.

			\item[3.] If $p \geq p_{JL}$ then for every $\alpha>0$ the operator $L_\alpha$ admits no positive eigenvalues.
		\end{itemize}
\end{theorem}
In particular the statement above implies, in the case $p<p_{JL}$ and $\eta<q_c,$ the existence of a forward self-similar solution that is linearly unstable in similarity variables for the deterministic unforced equation. Moreover, by choosing properly $\alpha$, the most unstable eigenvalue of $L_{\alpha}$ can be fixed to be arbitrarily small.
As a direct consequence of \autoref{thm_spectrum}, in \cite[Section 3]{glogic2025non} we proved the following result, which we shall also employ in the sequel.
	\begin{theorem}\label{thm:existence_ancient_solutions}
		Assume \begin{align*}
		    d\geq 3,\ 1+\frac{2}{d}<p<p_{JL},\ 1\leq \eta<q_c<\gamma, \gamma\geq p\eta.
		\end{align*}
        Let $\bar{\alpha}>0$ be such that for the corresponding expander $\bar{U}$ the operator $L_{\bar{\alpha}}: \mathcal{D}(L_{\bar{\alpha}}) \subseteq L^{\eta,\gamma} \rightarrow  L^{\eta,\gamma}$ admits a maximal positive eigenvalue $\la_{\bar{\alpha}}$, and  $\bar{U}^{lin}$ the eigenfunction of $L_{\bar{\alpha}}$ associated with $\lambda_{\bar{\alpha}}$. Then, for every sufficiently small $\eps>0$ there is $T<0$ such that there exists an ancient solution $\psi \in C_{loc}((-\infty,T],L^{\eta,\gamma})$ to 
        \begin{align*}
		\begin{cases}
		    \partial_\tau \psi=L_{\bar{\alpha}}\psi+\lvert \bar{U}+\psi\rvert^{p-1}(\bar{U}+\psi)-\lvert \bar{U}\rvert^{p-1}\bar{U}-p\lvert\bar{U}\rvert^{p-1}{\psi},\\
            \norm{\psi(\tau)}_{L^{\hat{q},\hat{r}}}\rightarrow 0 \quad \text{ as} \quad \tau\rightarrow -\infty
		\end{cases}
	\end{align*}
    for which 
		\begin{align*}
			\lVert \psi(\tau)\rVert_{L^{\eta,\gamma}}<\eps,
		\end{align*}
		and
		\begin{align*}
			\lVert \psi(\tau)\rVert_{L^{\gamma/p}_{rad}}> e^{\lambda_{\bar{\alpha}}\tau}\frac{\lVert \bar{U}^{lin}\rVert_{L^{\gamma/p}_{rad}}}{2},
		\end{align*}
		for  $\tau\in (-\infty,T]$.
	\end{theorem}
The idea behind the program of \cite{JiaSve15}, which we implemented in \cite{glogic2025non}, is that in case $\overline{\alpha}>0$ is such that $L_{\overline{\alpha}}$ has an unstable eigenvalue, then the nonlinear problem \eqref{Eq:Sim_var_w} admits a nontrivial solution arising from $0$ at $\tau=-\infty.$ Therefore, upon inverting the self-similar change of variables, both $U_{\overline{\alpha}}$ and $U_{\overline{\alpha}}+w$ are distributional solutions of \eqref{Eq:Heat}. They are not mild solutions due to poor integrability properties. However, if the most unstable eigenvalue is sufficiently small, a localization procedure can be performed to produce two mild solutions of \eqref{Eq:Heat} from the same initial data, thereby proving \autoref{Thm:main_deterministic}. Consequently, \autoref{Thm:main_additive_noise}, as we will describe in \autoref{proof_non_unique} below, reduces to showing that an analogous localization argument to that of \cite[Section 4]{glogic2025non} can also be carried out in the presence of a stochastic forcing that is white in time and sufficiently colored in space.
    
\section{non-uniqueness for the stochastically forced equation}\label{proof_non_unique}
	Let $1\leq q<q_c$ as in \autoref{Thm:main_additive_noise}. Let $\epsilon>0$ be such that
	\begin{align*}
		\frac{d(p-1)}{2q}-p<\epsilon<\frac{d(p-1)}{2q}-1.
	\end{align*}
	Then set $r=(p+\epsilon)q$ and $q_a=\frac{r}{p}.$ It follows immediately that \begin{align*}
	    1\leq q<q_a<q_c<r.
	\end{align*} 
    Moreover, due to \autoref{HP_noise_1} and consequently \eqref{eq:regularity_stoch_conv}, the stochastic convolution $z$ constructed in \autoref{Noise_not} satisfies
    \begin{align*}
        z\in C_{loc}([0,+\infty);L^{q,pr})\quad \mathbb{P}-a.s.,\quad z(0)=0.
    \end{align*}
	Let $\bar{U}$ and $\psi$ be from Theorem \ref{thm:existence_ancient_solutions} with $\eta=1$ and $\gamma=pr$. Then, for $t \in (0,e^T]$ we set
	\begin{align*}
		\tilde{u}_{1}(t,x)&=\frac{1}{t^{\frac{1}{p-1}}}\bar{U}\left(\frac{|x|}{\sqrt{t}}\right),\\ \quad \tilde{u}_{2}(t,x)&=\frac{1}{t^{\frac{1}{p-1}}}\bar{U}\left(\frac{|x|}{\sqrt{t}}\right)+\frac{1}{t^{\frac{1}{p-1}}}\psi\left(\ln t,\frac{|x|}{\sqrt{t}}\right).
	\end{align*}
	Recall that $\bar{U} \in C^2_{loc}([0,\infty))$ and 
	\begin{align*}
		\bar{U}(\rho)= O\left(\rho^{-\frac{2}{p-1}}\right) \quad \text{as} \quad  \rho \rightarrow \infty.
	\end{align*}
	\autoref{thm:existence_ancient_solutions} furthermore ensures that $\psi\in C_{loc}((-\infty,T],L^{1,pr})$ and
	\begin{align*}
		\lVert \psi(\tau)\rVert_{L^{1,pr}}<\eps \quad \text{and} \quad \lVert \psi(\tau)\rVert_{L^{r}_{rad}}>e^{\lambda_{\bar{\alpha}}\tau}\frac{\lVert \bar{U}^{lin}\rVert_{L^{r}_{rad}}}{2}
	\end{align*}
	for $\tau\in (-\infty,T]$. Both $\tilde{u}_1$ and $\tilde{u}_2$ are distributional solutions of \eqref{Eq:Heat} starting from the same initial data
    \begin{align*}
        \tilde{u}_0(x)=\frac{\ell(\overline{\alpha})}{\lvert x\rvert^{\frac{2}{p-1}} }.
    \end{align*}
    However, $\tilde{u}_0$ (and consequently also $\tilde{u}_1,\ \tilde{u}_2$) fails to be in $L^q(\R^d)$ due to poor integrability properties at $\lvert x\rvert\rightarrow +\infty.$ The idea for constructing two solutions of \eqref{stochastic_heat} is to define \begin{align*}
        u_0(x)=\tilde{u}_0(x)\one_{[0,\overline{R}]}(\lvert x\rvert),\ \overline{R}>0 
    \end{align*}
    and look for \begin{align*}
        u_1=z+\tilde{u}_1-w_1,\quad u_2=z+\tilde{u}_2-w_2
    \end{align*}
    for some $w_1,\ w_2$ which solve a certain nonlinear random PDE \eqref{nonlinear_pde} with initial condition $\tilde{u}_0-u_0.$ Then, a posteriori, we recover that $u_1$ and $u_2$ constructed in this way are two different mild solutions of \eqref{stochastic_heat}. We split the program shortly described above into \autoref{subsec_loc} and \autoref{sec:localization}.

\subsection{Localization Procedure}\label{subsec_loc}
Since both $\overline{u}(t,x):=\frac{1}{t^{\frac{1}{p-1}}}\overline{U}(\frac{x}{\sqrt{t}})$ and $\overline{u}(t,x)+u'(t,x),$  $u'(t,x)=\frac{1}{t^{\frac{1}{p-1}}}\psi(\log t,\frac{x}{\sqrt{t}})$ solve the deterministic nonlinear heat equation \eqref{Eq:Heat} and $z$ is a solution of the linear stochastic heat equation \eqref{Eq:Stoch_lin_heat}, we obtain that $w_{i},\ i\in\{1,2\},$ solve
\begin{align*}
    d w_{i}&=-d u_i+d \overline{u}+d u_i'+dz \\ & =\Delta w_i dt -\left[\lvert \overline{u}+u_i'-w_i+z\rvert^{p-1}(\overline{u}+u_i'-w_i+z)-\lvert \overline{u}+u_i'\rvert^{p-1}(\overline{u}+u_i')\right] dt \\ & + p\lvert \overline{u}\rvert^{p-1}(w_i-z)dt- p\lvert \overline{u}\rvert^{p-1}(w_i-z)dt, 
\end{align*}
where ${u}'_1=0,\ u'_2=u'$. 
In conclusion we are interested in studying the Cauchy problem
\begin{align}\label{nonlinear_pde}
    \begin{cases}
        \partial_t w&=\Delta w+p\lvert \overline{u}\rvert^{p-1}w+\tilde{f}(w),\\
        w(0)&=w_0
    \end{cases}
\end{align}
where
\begin{align}\label{forcing_nonlinear_aux}
    \tilde{f}(w)&=-\left[\lvert \overline{u}+u'-w+z\rvert^{p-1}(\overline{u}+u'-w+z)-\lvert \overline{u}+u'\rvert^{p-1}(\overline{u}+u')+ p\lvert \overline{u}\rvert^{p-1}(w-z)\right] \notag \\ &-p\lvert \overline{u}\rvert^{p-1} z
\end{align}
and $w_0\in L_{rad}^r(\R^d)$. Equation \eqref{nonlinear_pde} resembles the one treated in \cite[Section 4]{glogic2025non}. However, the nonlinearity is different and requires ad hoc treatment. Before stating the main result of this section we need to settle some notation. For each $T>0$, we define the Banach space
	\begin{align*}
		Z^{T}:=\{w\in L^{\infty}((0,T),&L^r_{rad}(\R^d)) ~\lvert \\
		& \sup_{t\in (0,T)}t^{\frac{d}{2r}\left(\frac{p-1}{p}\right)}\lVert w(t)\rVert_{ L^{pr}_{rad}}<\infty,~ \lim_{t\rightarrow 0}t^{\frac{d}{2r}\left(\frac{p-1}{p}\right)}\lVert w(t)\rVert_{ L^{pr}_{rad}}=0\},
	\end{align*}
	equipped with its natural norm
	\begin{align*}
		\| w \|_{Z^{T}}:=\sup_{t\in (0,T)}\lVert w(t)\rVert_{L^r_{rad}}+ \sup_{t\in (0,T)}t^{\frac{d}{2r}\left(\frac{p-1}{p}\right)}\lVert w(t)\rVert_{ L^{pr}_{rad}}.
	\end{align*}
    With this notation in mind, we are ready to prove the following well-posedness result for \eqref{non_linear_singular_pde_noise}.
\begin{proposition}\label{non_linear_singular_pde_noise}
		Let $\bar{\alpha}>0$ be such that for the corresponding expander $\bar{U}$ the operator $L_{\bar{\alpha}}: \mathcal{D}(L_{\bar{\alpha}}) \subseteq L^{1,pr} \rightarrow  L^{1,pr}$ admits a maximal positive eigenvalue $\la_{\bar{\alpha}}$ that satisfies 
        \begin{align}\label{item_instability}
		 {\lambda}_{\bar{\alpha}}<\frac{1}{p-1}-\frac{d}{2r}. 
		\end{align}
        Assume further that $w_0\in L^r_\text{rad}(\R^d)$ and $u':(0,1)\times \R^d \rightarrow \R$ is such that it can be written as $u'(t,x)=\frac{1}{t^{\frac{1}{p-1}}}U'\left(\ln t,\frac{|x|}{\sqrt{t}} \right)$, where
		\begin{align*}
			\sup_{\tau \in (-\infty,T]} \lVert U'(\tau,\cdot)\rVert_{L^{1,pr}}<\infty
		\end{align*}
		for some $T<0$. Then, whenever the above displayed quantity is sufficiently small, there is a stopping time $T'>0$ such that for $\mathbb{P}$-a.e. $\omega\in \Omega$ there exists a unique $w$ in $Z^{T'}$ solution to the Cauchy problem \eqref{non_linear_singular_pde_noise}-\eqref{forcing_nonlinear_aux}. Moreover, $w \in C([0,T'],L^r_{rad}(\R^d))\ \mathbb{P}$-a.s. and \begin{align}\label{energy_estimate_nonlin_1}
			\lVert w\rVert_{ L^{\infty}((0,T'),L^r_{rad})}+\operatorname{sup}_{t\in (0,T')}t^{\frac{d}{2r}\left(\frac{p-1}{p}\right)}\lVert w(t)\rVert_{ L^{pr}_{rad}}&\lesssim \lVert w_0\rVert_{L^r_{rad}}\quad \mathbb{P}-a.s.,
		\end{align}
		\begin{align}\label{energy_estimate_nonlin_2}
			\operatorname{lim}_{t\rightarrow 0}t^{\frac{d}{2r}\left(\frac{p-1}{p}\right)}\lVert w(t)\rVert_{ L^{pr}_{rad}}&=0\quad \mathbb{P}-a.s.    
		\end{align}
    
\end{proposition}
\begin{proof}
    As in \cite[Theorem 4.2]{glogic2025non}, we look for a solution of the form 
    \begin{align*}
        w(t)=\mathcal{S}[w_0,0](t)+\mathcal{S}[0,\tilde{f}(w)](t),
    \end{align*}
    where $\mathcal{S}(g,f)$ denotes the solution map for the linear problem 
\begin{align*}
    \begin{cases}
        \partial_t w&=\Delta w+p\lvert \overline{u}\rvert^{p-1}w+f,\\
        w(0)&=g
    \end{cases}
\end{align*}
    given by \cite[Lemma 4.1]{glogic2025non}.  Let us denote by $M=\norm{\mathcal{S}[w_0,0]}_{Z^{e^T}}$ and by $B_{2M}\subseteq Z^{T'}$ the closed ball in $Z^{T'}$ for $T'>0$ with center $0$ and radius $2M$. We are looking for $\eps>0$ and $T'>0$ small enough such that 
		\begin{align*}
			\Gamma(w)=\mathcal{S}[w_0,0]+\mathcal{S}[0,\tilde{f}(w)]
		\end{align*}
		is a contraction on $B_{2M}$. First we need to show that $\Gamma$ maps $B_{2M}$ into itself.
        We start by observing that 
        \begin{align*}
            \tilde{f}(w)&=f(w)-p\lvert \overline{u}\rvert^{p-1}z,\\ f(w)&=\lvert \overline{u}+u'-w+z\rvert^{p-1}(\overline{u}+u'-w+z)-\lvert \overline{u}+u'\rvert^{p-1}(\overline{u}+u')- p\lvert \overline{u}\rvert^{p-1}(w-z).
        \end{align*}
        Moreover, by H\"older’s inequality and the definition of $\overline{u}$, we have
\begin{align}\label{estimate_1_nonlinear_aux}
      \norm{\lvert \overline{u}(t)\rvert^{p-1}z(t)}_{L^{q_a}_{rad}}&\leq t^{\frac{d}{2q_a}-\frac{d}{2r}-1} \norm{ \overline{U}}_{L^r_{rad}}^{p-1}\norm{z(t)}_{L^r_{rad}},\\
      \label{estimate_2_nonlinear_aux}
      \norm{\lvert \overline{u}(t)\rvert^{p-1}z(t)}_{L^{r}_{rad}}&\leq t^{-1} \norm{ \overline{U}}_{L^{\infty}_{rad}}^{p-1}\norm{z(t)}_{L^r_{rad}}.
    \end{align}
    Since $z(t)\in C([0,e^T];L^r_{rad})$ and $z(0)=0$ we have also
    \begin{align*}
        \lim_{t \rightarrow 0} \left(\norm{\lvert \overline{u}(t)\rvert^{p-1}z(t)}_{L^{r}_{rad}}t+\norm{\lvert \overline{u}(t)\rvert^{p-1}z(t)}_{L^{q_a}_{rad}} t^{1+\frac{d}{2r}-\frac{d}{2q_a}}\right)=0.
    \end{align*}
    Now we estimate the other terms in $\tilde{f}(w)$ distinguishing between the case $p\in (1+\frac{2}{d},2]$ and $p>2$. In the first case, due to \cite[Lemma 3.2]{glogic2025non} with $x=\overline{u},\ y=u',\ z=u'-w+z$
    we have
    \begin{align*}
\lVert f(w)\rVert_{L^{q_a}_{rad}}&\lesssim_p \lVert  w\rVert_{L^r_{rad}}^p+\lVert  z\rVert_{L^r_{rad}}^p+\lVert \lvert u'\rvert^{p-1}\lvert w\rvert\rVert_{L^{q_a}_{rad}}+\lVert \lvert u'\rvert^{p-1}\lvert z\rvert\rVert_{L^{q_a}_{rad}},\\  \lVert f(w)\rVert_{L^r_{rad}}&\lesssim_p \lVert w\rVert_{L^{rp}_{rad}}^p+\lVert z\rVert_{L^{rp}_{rad}}^p+\lVert \lvert u'\rvert^{p-1}\lvert w\rvert\rVert_{L^r_{rad}}+\lVert \lvert u'\rvert^{p-1}\lvert z\rvert\rVert_{L^r_{rad}}.
    \end{align*}
		Therefore, by exploiting relations \cite[equations (4.25)-(4.26)-(4.35)]{glogic2025non}, we obtain, thanks to the fact that $w\in B_{2M}$ and  $z(t)\in C([0,e^T];L^{r,pr})$ with $z(0)=0$,
		\begin{align*}
			&t^{1+\frac{d}{2r}-\frac{d}{2{q_a}}}\lVert f(w(t))\rVert_{L^{q_a}_{rad}}\lesssim \eps^{p-1} (M+\norm{z}_{C([0,T'];L^r_{rad})})+\left(T'\right)^{1-\frac{d(p-1)}{2r}}(M^{p}+\norm{z}_{C([0,T'];L^r_{rad})}^p),\\
			&t\lVert f(w(t))\rVert_{L^r_{rad}} \lesssim \eps^{p-1}\left(M+\norm{z}_{C([0,T'];L^{pr})}\right)+T'\norm{z}_{C([0,T'];L^{pr})}^p+\left(T'\right)^{1-\frac{d(p-1)}{2r}}M^p,\\
			&\operatorname{lim}_{t\rightarrow 0} t^{1+\frac{d}{2r}-\frac{d}{2{q_a}}}\lVert f(w(t))\rVert_{L^{q_a}_{rad}}=0,\quad \operatorname{lim}_{t\rightarrow 0} t\lVert f(w(t))\rVert_{L^r_{rad}}=0.
		\end{align*}  
        Therefore $\tilde{f}(w)$ satisfies the assumptions of \cite[Lemma 4.2]{glogic2025non}
        and 
        \begin{align*}
			\lVert  \Gamma(w)\rVert_{Z^{T'}}&\leq M+C\left\{\norm{z}_{C([0,T'];L^{r,pr}_{rad})}+\eps^{p-1} M +\left(T'\right)^{1-\frac{d(p-1)}{2r}}(M^{p}+\norm{z}_{C([0,T'];L^{r,pr}_{rad})}^p)\right\}\\ & \leq 2M,
		\end{align*}
		whenever $\eps$ and $T'$ satisfy 
        \begin{align*}
            & C\epsilon^{p-1}\leq 1/2,\\
            & T'=\inf\left\{t>0:C\left(\norm{z}_{C([0,t];L^{r,pr}_{rad})}+t^{1-\frac{d(p-1)}{2r}}(M^{p}+\norm{z}_{C([0,t];L^{r,pr}_{rad})}^p)\right)\geq M/2\right\}\wedge e^T.
        \end{align*}
        We are using again the continuity of $z$ and the fact that $z(0)=0$ to define $T'$.
In case of $p>2$, due to \cite[Lemma 3.2]{glogic2025non} with $x=\overline{u},\ y=u',\ z=u'-w+z$ and Young's inequality we have
    \begin{align*}
\lVert f(w)\rVert_{L^{q_a}_{rad}}\lesssim_p \, & \lVert w\rVert_{L^r_{rad}}^p+\lVert z\rVert_{L^r_{rad}}^p+\lVert \lvert u'\rvert^{p-1}\lvert w\rvert\rVert_{L^{q_a}_{rad}}+\lVert \lvert u'\rvert^{p-1}\lvert z\rvert\rVert_{L^{q_a}_{rad}} \\ & +\lVert\lvert \bar{u}\rvert^{p-2}\lvert u'\rvert\lvert w\rvert \rVert_{L^{q_a}_{rad}}+\lVert\lvert \bar{u}\rvert^{p-2}\lvert u'\rvert\lvert z\rvert \rVert_{L^{q_a}_{rad}}+\lVert\lvert \bar{u}\rvert^{p-2}\lvert w\rvert^2 \rVert_{L^{q_a}_{rad}}+\lVert\lvert \bar{u}\rvert^{p-2}\lvert z\rvert^2 \rVert_{L^{q_a}_{rad}},\\  \lVert f(w)\rVert_{L^r_{rad}}\lesssim_p \, & \lVert w\rVert_{L^{pr}_{rad}}^p+\lVert z\rVert_{L^{pr}_{rad}}^p+\lVert \lvert u'\rvert^{p-1}\lvert w\rvert\rVert_{L^{r}_{rad}}+\lVert \lvert u'\rvert^{p-1}\lvert z\rvert\rVert_{L^{r}_{rad}} \\ & +\lVert\lvert \bar{u}\rvert^{p-2}\lvert u'\rvert\lvert w\rvert \rVert_{L^{r}_{rad}}+\lVert\lvert \bar{u}\rvert^{p-2}\lvert u'\rvert\lvert z\rvert \rVert_{L^{r}_{rad}}+\lVert\lvert \bar{u}\rvert^{p-2}\lvert w\rvert^2 \rVert_{L^{r}_{rad}}+\lVert\lvert \bar{u}\rvert^{p-2}\lvert z\rvert^2 \rVert_{L^{r}_{rad}}.
    \end{align*}
		Therefore, by exploiting relations \cite[equations (4.25)-(4.26)-(4.31)-(4.32)-(4.33)-(4.34)-(4.35)-(4-36)]{glogic2025non}, we obtain thanks to the fact that $w\in B_{2M}$ and  $z(t)\in C([0,e^T];L^{r,pr})$ with $z(0)=0$
		\begin{align*}
		t^{1+\frac{d}{2r}-\frac{d}{2{q_a}}}\lVert f(w(t))\rVert_{L^{q_a}_{rad}}&\lesssim \eps \left(M+\norm{z}_{C([0,T'];L^r_{rad})}\right)\\ &+\left(T'\right)^{\frac{1}{p-1}-\frac{d}{2r}}(M^2+M^{p}+\norm{z}_{C([0,T'];L^r_{rad})}^2+\norm{z}_{C([0,T'];L^r_{rad})}^p),\\
			t\lVert f(w(t))\rVert_{L^r_{rad}} \lesssim & \eps \left(M+ \norm{z}_{C([0,T'];L^{pr})}\right)\\ &+\left(T'\right)^{\frac{1}{p-1}-\frac{d}{2r}}(M^2+M^{p}+\norm{z}_{C([0,T'];L^{pr}_{rad})}^2+\norm{z}_{C([0,T'];L^{pr}_{rad})}^p),\\
			\operatorname{lim}_{t\rightarrow 0} t^{1+\frac{d}{2r}-\frac{d}{2{q_a}}}\lVert f(w(t))\rVert_{L^{q_a}_{rad}}&=0,\quad \operatorname{lim}_{t\rightarrow 0} t\lVert f(w(t))\rVert_{L^r_{rad}}=0.
		\end{align*}  
        Therefore, in this case as well, $\tilde{f}(w)$ satisfies the assumptions of \cite[Lemma 4.2]{glogic2025non}
        and 
        \begin{align*}
			\lVert  \Gamma(w)\rVert_{Z^{T'}}&\leq M\\ & +C\left\{\norm{z}_{C([0,T'];L^{r,pr}_{rad})}+\eps M +\left(T'\right)^{\frac{1}{p-1}-\frac{d}{2r}}(M^2+M^{p}+\norm{z}_{C([0,T'];L^{r,pr}_{rad})}^2+\norm{z}_{C([0,T'];L^{r,pr}_{rad})}^p)\right\}\\ & \leq 2M,
		\end{align*}
		whenever $\eps$ and $T'$ are chosen such that
        \begin{align*}
            & C\epsilon \leq 1/2,\\
            & T'=\inf\left\{t>0:C\left\{\norm{z}_{C([0,t];L^{r,pr}_{rad})} +t^{\frac{1}{p-1}-\frac{d}{2r}}(M^2+M^{p}+\norm{z}_{C([0,t];L^{r,pr}_{rad})}^2+\norm{z}_{C([0,t];L^{r,pr}_{rad})}^p)\right\} \leq M/2\right\}\\ &\quad \quad\quad \quad \wedge e^T.
        \end{align*}
        		Now we show that $\Gamma$ is a contraction in $B_{2M}$, possibly after reducing $\eps$ and $T'$. For $w_1,\ w_2\in B_{2M}$, we observe that
		\begin{align*}
			\Gamma(w_1)-\Gamma(w_2)=\mathcal{S}[0,f(w_1)-f(w_2)].
		\end{align*} We first consider the case $p\in (1+\frac{2}{d},2]$. Thanks to \cite[Lemma 3.2]{glogic2025non} with $x=\bar{u},\ y=u'-w_1+z,\ z=u'-w_2+z$, we get
		\begin{align*}
			\lVert f(w_1)-f(w_2)\rVert_{L^{q_a}_{rad}}\leq \, &  \lVert \lvert u' \rvert^{p-1}\lvert w_1-w_2\rvert \rVert_{L^{q_a}_{rad}}\\ &+ \left(\lVert w_1 \rVert^{p-1}_{L^r_{rad}}+\lVert w_2 \rVert^{p-1}_{L^r_{rad}}+2\lVert z \rVert^{p-1}_{L^r_{rad}}\right)\lVert w_1-w_2\rVert_{L^r_{rad}},\\
			\lVert f(w_1)-f(w_2)\rVert_{L^r_{rad}}\leq \, &  \lVert \lvert u' \rvert^{p-1}\lvert w_1-w_2\rvert \rVert_{L^r_{rad}}\\ &+ \left(\lVert w_1 \rVert^{p-1}_{L^{pr}_{rad}}+\lVert w_2 \rVert^{p-1}_{L^{pr}_{rad}}+2\lVert z \rVert^{p-1}_{L^{pr}_{rad}}\right)\lVert w_1-w_2\rVert_{L^{pr}_{rad}}.
		\end{align*}
		Therefore, \cite[equations (4.25)-(4.26)-(4.35)]{glogic2025non} imply, thanks to the fact that $w_1,w_2\in B_{2M}$ and $z\in C([0,T];L^{r,pr})$,
		\begin{align*}
			t^{1+\frac{d}{2r}-\frac{d}{2{q_a}}}\lVert f(w_1(t))-f(w_2(t))\rVert_{L^{q_a}_{rad}}&\lesssim\eps^{p-1} M+\left(T'\right)^{1-\frac{d(p-1)}{2r}}\left(M^{p}+M \norm{z}_{C([0,T'];L^r)}^{p-1}\right),\\
			t\lVert f(w_1(t))-f(w_2(t))\rVert_{L^r_{rad}}& \lesssim\eps^{p-1}M+\left(T'\right)^{1-\frac{d(p-1)}{2r}}\left(M^{p}+M \norm{z}_{C([0,T'];L^{pr})}^{p-1}\right),\\
			\operatorname{lim}_{t\rightarrow 0} t^{1+\frac{d}{2r}-\frac{d}{2{q_a}}}\lVert f(w_1(t))-f(w_2(t))\rVert_{L^{q_a}_{rad}}&=0,\\ \operatorname{lim}_{t\rightarrow 0} t\lVert f(w_1(t))-f(w_2(t))\rVert_{L^r_{rad}}&=0.
		\end{align*}
        Since $f(w_1)-f(w_2)$ satisfies the assumptions of \cite[Lemma 4.2]{glogic2025non}
        we obtain
		\begin{align*}
			\lVert  \Gamma(w_1)-\Gamma(w_2)\rVert_{Z^{T'}}&\lesssim\left(\eps^{p-1}+\left(T'\right)^{1-\frac{d(p-1)}{2r}}\left(M^{p-1}+\norm{z}_{C([0,T'];L^{r,pr})}^{p-1}\right)\right)\lVert w_1-w_2\rVert_{Z^{T'}}\\ & \leq \frac{1}{2}\lVert w_1-w_2\rVert_{Z^{T'}}
		\end{align*}
		for $\eps$ and $T'$ small enough. Namely, denoting by $C'$ the hidden constant in the above inequality, we require, in this case, that
        \begin{align*}
            C' \eps^{p-1}&\leq 1/4
            \\
            T'&=\inf\left\{t>0:C\left(\norm{z}_{C([0,t'];L^{r,pr}_{rad})}+t^{1-\frac{d(p-1)}{2r}}(M^{p}+\norm{z}_{C([0,t];L^{r,pr}_{rad})}^p)\right)\geq M/2\right\}\wedge e^T\\ & \wedge 
\inf\left\{t>0:C' t^{1-\frac{d(p-1)}{2r}}\left(M^{p-1}+\norm{z}_{C([0,t];L^{pr})}^{p-1}\right) > 1/4\right\}.
        \end{align*}
In the case $p>2$, by \cite[Lemma 3.2]{glogic2025non} with $x=\bar{u},\ y=u'-w_1+z,\ z=u'-w_2+z$, and using Young's inequality, we obtain
		\begin{align*}
			\lVert f(w_1)-f(w_2)\rVert_{L^{q_a}_{rad}}\lesssim_p\,&   \lVert \lvert u' \rvert^{p-1}\lvert w_1-w_2\rvert \rVert_{L^{q_a}_{rad}}\\ &+ \left(\lVert z \rVert^{p-1}_{L^r_{rad}}+\lVert w_1 \rVert^{p-1}_{L^r_{rad}}+\lVert w_2 \rVert^{p-1}_{L^r_{rad}}\right)\lVert w_1-w_2\rVert_{L^r_{rad}}\\ & +\lVert \lvert \bar{u}\rvert^{p-2}\lvert u'\rvert \lvert w_1-w_2\rvert \rVert_{L^{q_a}_{rad}}\\ & +\lVert \lvert \bar{u}\rvert^{p-2}\lvert z\rvert \lvert w_1-w_2\rvert \rVert_{L^{q_a}_{rad}}\\ & +\lVert \lvert \bar{u}\rvert^{p-2}\left(\lvert w_1\rvert+\lvert w_2\rvert\right) \lvert w_1-w_2\rvert \rVert_{L^{q_a}_{rad}},\\
			\lVert f(w_1)-f(w_2)\rVert_{L^r_{rad}} \lesssim_p\,&   \lVert \lvert u' \rvert^{p-1}\lvert w_1-w_2\rvert \rVert_{L^{r}_{rad}}\\ &+ \left(\lVert z \rVert^{p-1}_{L^{pr}_{rad}}+\lVert w_1 \rVert^{p-1}_{L^{pr}_{rad}}+\lVert w_2 \rVert^{p-1}_{L^{pr}_{rad}}\right)\lVert w_1-w_2\rVert_{L^{pr}_{rad}}\\ & +\lVert \lvert \bar{u}\rvert^{p-2}\lvert u'\rvert \lvert w_1-w_2\rvert \rVert_{L^{r}_{rad}}\\ & +\lVert \lvert \bar{u}\rvert^{p-2}\lvert z\rvert \lvert w_1-w_2\rvert \rVert_{L^{r}_{rad}}\\ & +\lVert \lvert \bar{u}\rvert^{p-2}\left(\lvert w_1\rvert+\lvert w_2\rvert\right) \lvert w_1-w_2\rvert \rVert_{L^{r}_{rad}}.
		\end{align*}
				Therefore, by exploiting relations \cite[equations (4.25)-(4.26)-(4.31)-(4.32)-(4.33)-(4.34)-(4.35)-(4-36)]{glogic2025non}, we obtain thanks to the fact that $w\in B_{2M}$ and  $z(t)\in C([0,e^T];L^{r,pr})$ with $z(0)=0$ 
		\begin{align*}
			t^{1+\frac{d}{2r}-\frac{d}{2{q_a}}}&\lVert f(w_1(t))-f(w_2(t))\rVert_{L^{q_a}_{rad}}\\ &\lesssim \eps M+(T')^{\frac{1}{p-1}-\frac{d}{2r}}M(M+M^{p-1}+\norm{z}_{C([0,T'];L^r)}+\norm{z}_{C([0,T'];L^r)}^{p-1}),\\
			t&\lVert f(w_1(t))-f(w_2(t))\rVert_{L^r_{rad}}\\ &\lesssim \eps M+\left(T'\right)^{\frac{1}{p-1}-\frac{d}{2r}}M\left(M+M^{p-1}+ \norm{z}_{C([0,T'];L^{pr})}+\norm{z}_{C([0,T'];L^{pr})}^{p-1}\right),\end{align*}
		\begin{align*}
			&\operatorname{lim}_{t\rightarrow 0} t^{1+\frac{d}{2r}-\frac{d}{2{q_a}}}\lVert f(w_1(t))-f(w_2(t))\rVert_{L^{q_a}}=0,\\  &\operatorname{lim}_{t\rightarrow 0} t\lVert f(w_1(t))-f(w_2(t))\rVert_{L^r_{rad}}=0.
		\end{align*}
		We can then apply \cite[Lemma 4.1]{glogic2025non} to conclude that
		\begin{multline*}
			\lVert  \Gamma(w_1)-\Gamma(w_2)\rVert_{Z^{T'}}\\\lesssim
            \left(\eps+(T')^{\frac{1}{p-1}-\frac{d}{2r}}(M+M^{p-1}+\norm{z}_{C([0,T'];L^{r,pr})}+\norm{z}_{C([0,T'];L^{r,pr})}^{p-1})\right)
            \lVert w_1-w_2\rVert_{Z^{T'}}\\  \leq \frac{1}{2}\lVert w_1-w_2\rVert_{Z^{T'}},
		\end{multline*}
		for $\eps$ and $T'$ small enough. 
        Namely, calling $C'$ the hidden constant in the inequality above, in this case we require 
        \begin{align*}
            C' \eps^{p-1}&\leq 1/4
            \\
            T'&=\inf\left\{t>0:C\left\{\norm{z}_{C([0,t];L^{r,pr}_{rad})} +t^{\frac{1}{p-1}-\frac{d}{2r}}(M^2+M^{p}+\norm{z}_{C([0,t];L^{r,pr}_{rad})}^2+\norm{z}_{C([0,t];L^{r,pr}_{rad})}^p)\right\} \leq M/2\right\}\\ & \wedge 
\inf\left\{t>0:C't^{\frac{1}{p-1}-\frac{d}{2r}}(M+M^{p-1}+\norm{z}_{C([0,t];L^{r,pr})}+\norm{z}_{C([0,t];L^{r,pr})}^{p-1})> \frac{1}{4}\right\}\wedge e^T.
\end{align*}
This completes the proof.    
\end{proof}
As the proof above shows, by solving the Cauchy problem for $w_1$ and $w_2$, we can construct each solution up to the random time
\begin{align*}
\mathcal{T}_i:=T'_i=e^T\wedge T_i \wedge \inf_{t\geq 0}\{\norm{z}_{C([0,t];L^{r,pr})}\geq C_i\},\ i\in\{1,2\}.
\end{align*}
Here, $T_i$ and $C_i$ are deterministic constants depending on $i$, $p$, and $r$, as well as on the deterministic initial condition of the PDE, which is independent of $i\in{1,2}$. Since $z$ has continuous paths in $L^{q,pr}$, $\mathcal{T}_i$ are stopping times. 
The $w_i$, as defined in \autoref{non_linear_singular_pde_noise}, are progressively measurable processes on $(0,+\infty)$ with values in $L^r \cap L^{pr}$, once extended by setting $w_i(t) \equiv w_i(\mathcal{T}_i)$ for $t \geq \mathcal{T}_i$.
 Moreover, due to the continuity of $w_i$ in $L^r$, the $w_i$ are defined for all times and consequently adapted processes on $L^r$. The choice of the extension of the $w_i$ after the stopping times is not the only possible one. However, in the following, we will show our non-uniqueness by considering the two solutions before $\mathcal{T}_1\wedge \mathcal{T}_2$. Therefore, the latter choice does not affect the proof of our result, which is the content of the next subsection.
\subsection{Proof of \autoref{Thm:main_additive_noise}}\label{sec:localization}
     Define $w_0:=\tilde{u}_0-u_0$. We can now invoke \autoref{thm:existence_ancient_solutions} and \autoref{non_linear_singular_pde_noise} to obtain two processes
	\begin{align*}
		u_1(t,x)&=z(t,x)+\frac{1}{t^{\frac{1}{p-1}}}\bar{U}\left(\frac{|x|}{\sqrt{t}}\right) - w_1(t,x),\\ u_{2}(t,x)&=z(t,x)+\frac{1}{t^{\frac{1}{p-1}}}\bar{U}\left(\frac{|x|}{\sqrt{t}}\right)+\frac{1}{t^{\frac{1}{p-1}}}\psi\left(\ln t,\frac{|x|}{\sqrt{t}}\right) - w_2(t,x),
	\end{align*}
	which are weak solutions of \eqref{stochastic_heat} on $(0,\mathcal{T}_1\wedge \mathcal{T}_2)$ for $\mathcal{T}_1$ and $\mathcal{T}_2$ stopping times defined at the end of the previous subsection. In particular, one can easily check that $u_1, u_2\in L^p_{{loc}}((0,\mathcal{T}_1\wedge \mathcal{T}_2)\times \R^d)\quad \mathbb{P}-a.s$. We claim that they are, in fact, two different local mild $L^q$-solutions on $[0,\mathcal{T}_1\wedge \mathcal{T}_2]$ with the same initial datum $u_0$. By construction, $u_0$ is compactly supported, and therefore in $L^q(\R^d)$. It remains to show that $u_1, u_2\in C([0,\mathcal{T}_1\wedge\mathcal{T}_2],L^q(\R^d))\quad \mathbb{P}-a.s.$, they have the required measurability properties and $u_1\neq u_2.$ We show the first property in the form of a lemma.
	
	\begin{lemma}\label{Continuity_uniform_bound}
		Let $i\in \{1,2\}$, $u_i\in C([0,\mathcal{T}_i],L^q(\R^d))\quad \mathbb{P}-a.s.$ for each $1\leq q<q_c$. Moreover $u_i(\cdot\wedge \mathcal{T}_i)$ is an adapted process with values in $L^q(\R^d)$.  
	\end{lemma}
	
	\begin{proof}
        Since $z\in C_{loc}([0,+\infty);L^{q,pr}(\R^d))$ and has the required measurability properties, it is enough to show that $v_i:=u_i-z$ belongs to $ C([0,\mathcal{T}_i],L^q(\R^d))$ and that the stopped process is adapted.
		We prove the lemma only for $v_2$, the other case being analogous and simpler. We therefore write, for convenience, $v$ (resp.~$w,\mathcal{T}$) in place of $v_2$ (resp.~$w_2,\ \mathcal{T}_2$). Let us introduce a radially symmetric cut off $\chi\in C^{\infty}_c(\R^d)$ by
		\begin{align*}
			\chi(x):=\begin{cases}
				1\quad \text{if } \lvert x\rvert\leq 1\\
				0\quad \text{if } \rvert x\rvert\geq 2,
			\end{cases}\quad \chi(x)\geq 0,\quad \chi(x)\leq \chi(x') \quad \text{if} \quad \lvert x\rvert\geq \lvert x'\rvert.
		\end{align*}
		Furthermore, for $x_0\in \R^d,\, R\geq 1$, define $\chi_{x_0, R}(x):=\chi\left(\frac{x-x_0}{R}\right).$ Due to the properties of $\bar{U},\psi, w$, we have that for each $x_0\in \R^d,\, R\geq 1$ \begin{align}\label{regularity_cutoff}
			v\chi_{x_0,R}\in C([0,\mathcal{T}],L^q(\R^d))\cap C((0,\mathcal{T}],L^r(\R^d))\quad \mathbb{P}-a.s.,
		\end{align}            
		and for each $t,t_0\in (0,\mathcal{T}]$ with $t_0\leq t$, that
		\begin{align*}
			v(t)\chi_{x_0,R}= \, &e^{(t-t_0)\Delta}v(t_0)\chi_{x_0,R}+\int_{t_0}^t e^{(t-s)\Delta}\left(v(s)\Delta\chi_{x_0,R}-2\operatorname{div}(v(s)\nabla \chi_{x_0,R})\right) ds\\ & +\int_{t_0}^t e^{(t-s)\Delta} \chi_{x_0,R} \lvert v(s)+z(s)\rvert^{p-1}(v(s)+z(s)) ds,
		\end{align*}
		in  $L^q(\R^d)$, where $e^{t\Delta}$ is the heat semigroup on $\R^d$.
		Considering the $L^q$ norm of the equation above, by triangle inequality and the definition of $\chi_{x_0,R}$ we get easily  that
		\begin{align*}
			\norm{ v(t)}_{L^q(B_{R}(x_0))}\lesssim \, & \norm{v(t_0)\chi_{x_0,R}}_{L^q(\R^d)}+\int_{t_0}^t \left(1+\frac{1}{\sqrt{t-s}}\right) \norm{ v(s)}_{L^q(B_{2R}(x_0))}ds\\ & +\int_{t_0}^t\norm{ \chi_{x_0,R} \lvert v(s)+z(s)\rvert^{p-1}(v(s)+z(s))}_{L^q(\R^d)} ds ,
		\end{align*}
		the hidden constant above being independent of $R$, $t_0,$ $t$ and $x_0$. Let us analyze further the last term.
		By interpolation and Young's inequality, since $r>pq$, we have
		\begin{align*}
			\norm{ \chi_{x_0,R} \lvert v(s)+z(s)\rvert^{p-1}(v(s)+z(s))}_{L^q(\R^d)}\lesssim \, & \norm{ \chi^{1/p}_{x_0,R} v(s)}_{L^{pq}(\R^d)}^p+\norm{z(s)}_{L^{q,r}(\R^d)}^p\\ \leq \, & \norm{ \chi^{1/p}_{x_0,R} v(s)}_{L^{q}(\R^d)}^{\frac{r-pq}{r-q}}\norm{  v(s)}_{L^{r}(\R^d)}^{\frac{r(p-1)}{r-q}}+\norm{z(s)}_{L^{q,r}(\R^d)}^p\\  \lesssim \, & \norm{ \chi^{1/p}_{x_0,R} v(s)}_{L^{q}(\R^d)}+\norm{  v(s)}_{L^{r}(\R^d)}^{\frac{r}{q}}+\norm{z(s)}_{L^{q,r}(\R^d)}^p\\ \lesssim \, & \norm{  v(s)}_{L^{q}(B_{2R}(x_0))}+\frac{1}{s^{\frac{r}{q}\left(\frac{1}{p-1}-\frac{d}{2r}\right)}}\left(\norm{\bar{U}}_{L^r(\R^d)}^{\frac{r}{q}}+\eps^{\frac{r}{q}}\right)\\ & +\norm{w}_{C([0,\mathcal{T}],L^r(\R^d))}^{\frac{r}{q}}+\norm{z}_{C([0,\mathcal{T}];L^{q,r}(\R^d))}^p.
		\end{align*}
		Therefore 
		\begin{align*}
			\norm{ v(t)}_{L^q(B_{R}(x_0))}\lesssim \,& \norm{v(t_0)\chi_{x_0,R}}_{L^q(\R^d)}+\int_{t_0}^t \left(1+\frac{1}{\sqrt{t-s}}\right) \norm{ v(s)}_{L^q(B_{2R}(x_0))}ds\\ & +\int_{t_0}^t 1+\frac{1}{s^{\frac{r}{q}\left(\frac{1}{p-1}-\frac{d}{2r}\right)}} ds.
		\end{align*}
		Since $u_0\in L^q(\R^d)$, $\frac{r}{q}\left(\frac{1}{p-1}-\frac{d}{2r}\right)<1$, and the relation \eqref{regularity_cutoff} holds, we can let $t_0\rightarrow 0$ to get
		\begin{align*}
			\norm{ v(t)}_{L^q(B_{R}(x_0))}&\lesssim \left(\norm{u_0}_{L^q(\R^d)}+1\right)+\int_{0}^t \left(1+\frac{1}{\sqrt{t-s}}\right) \norm{ v(s)}_{L^q(B_{2R}(x_0))}ds.   
		\end{align*}
		Now, let us introduce the function \begin{align*}
			v_R(t)=\sup_{x_0\in \R^d}\norm{ v(t)}_{L^q(B_{R}(x_0))}.
		\end{align*}
		Obviously
		\begin{align*}
			\norm{ v(t)}_{L^q(B_{2R}(x_0))}\lesssim v_R(t)
		\end{align*}
		and 
		\begin{align*}
			\norm{ v(t)}_{L^q(B_{R}(x_0))}  &\lesssim \left(\norm{u_0}_{L^q(\R^d)}+1\right)+\int_{0}^t \left(1+\frac{1}{\sqrt{t-s}}\right) v_R(s)ds.  
		\end{align*}
		Taking the supremum in $x_0$ of the expression above and applying Gr\"onwall's inequality, we get 
		\begin{align*}
			v_R(t)\lesssim \norm{u_0}_{L^q(\R^d)}+1,
		\end{align*}
		for all $t \in [0,\mathcal{T}]$.
		By letting $R\rightarrow +\infty$, we get that $v\in L^{\infty}((0,\mathcal{T}),L^q(\R^d))\quad \mathbb{P}-a.s.$ and consequently the same applies also to $u$. In order to show the continuity of $u(t)$ in $L^q$, we use again the mild formulation for $v(t)$. Namely, for each $0\leq t_1\leq t_2\leq \mathcal{T}$, we have that
		\begin{align*}
			v(t_2)\chi_{x_0,R}-v(t_1)&=\left(e^{(t_2-t_1)\Delta}v(t_1)\chi_{x_0,R}-v(t_1)\right)\\ &+\int_{t_1}^{t_2} e^{(t-s)\Delta}\left(v(s)\Delta\chi_{x_0,R}-2\operatorname{div}(v(s)\nabla \chi_{x_0,R})\right) ds\\ & +\int_{t_1}^{t_2} e^{(t-s)\Delta} \chi_{x_0,R} \lvert v(s)+z(s)\rvert^{p-1}(v(s)+z(s)) ds,
		\end{align*}
		in $L^q(\R^d)$.   
		Taking the $L^q$-norm of the expression above and letting $R\rightarrow +\infty$, we obtain, by analogous considerations to the ones employed to obtain the uniform bound on the $L^q$-norm
		\begin{align*}
			\norm{v(t_2)-v(t_1)}_{L^q(\R^d)} = \, &\limsup_{R\rightarrow +\infty}\norm{v(t_2)\chi_{x_0,R}-v(t_1)}_{L^q(\R^d)}\\  \leq \, & \norm{e^{(t_2-t_1)\Delta}v(t_1)-v(t_1)}_{L^q(\R^d)}\\ & +\sup_{R\geq 1}\norm{\int_{t_1}^{t_2} e^{(t-s)\Delta}\left(v(s)\Delta\chi_{x_0,R}-2\operatorname{div}(v(s)\nabla \chi_{x_0,R})\right)ds}_{L^q(\R^d)}\\ & +\sup_{R\geq 1}\norm{\int_{t_1}^{t_2} e^{(t-s)\Delta} \chi_{x_0,R} \lvert v(s)+z(s)\rvert^{p-1}(v(s)+z(s)) ds }_{L^q(\R^d)}\\  \lesssim\, & \norm{e^{(t_2-t_1)\Delta}v(t_1)-v(t_1)}_{L^q(\R^d)}\\ & +\int_{t_1}^{t_2} \left(1+\frac{1}{\sqrt{t_2-s}}\right) \norm{ v(s)}_{L^q(\R^d)}ds+\int_{t_1}^{t_2} 1+\frac{1}{s^{\frac{r}{q}\left(\frac{1}{p-1}-\frac{d}{2r}\right)}} ds.
		\end{align*}
		Since we already proved that $v\in L^{\infty}((0,\mathcal{T}_1\wedge\mathcal{T}_2),L^q(\R^d))\quad \mathbb{P}-a.s.,$ the continuity of $v$ follows from the last inequality, and consequently, so does the continuity of $u$.” \\
        Concerning the measurability, due to the continuity of $u\in C([0,\mathcal{T}];L^q)$, it is enough to show that for each $t\in [0,+\infty),\ v(t\wedge \mathcal{T})$ is $\mathcal{F}_t$ measurable. Therefore, according to \cite[Proposition 1.1]{da2014stochastic}, it is enough to show that for each $ \phi\in L^{q'}(\R^d)$, the real-valued random variable $v_{t,\phi}:=\int_{\R^d} v(t\wedge \mathcal{T},x)\phi(x)dx$ is $\mathcal{F}_t$-measurable.
        Due to the comments below the proof of \autoref{non_linear_singular_pde_noise}, we already know that the latter holds for any $\phi\in L^{r'}(\R^d)$. However, by dominated convergence
        \begin{align*}
         v_{t,\phi}=\lim_{r\rightarrow +\infty}  v_{t,\phi\chi_{x_0,r}}\quad \mathbb{P}-a.s.
        \end{align*}
        Since $\phi\chi_{x_0,r}\in L^{r'}$, $v_{t,\phi\chi_{x_0,r}}$ is $\mathcal{F}_t$-measurable. Therefore,  $v_{t,\phi}$ is also $\mathcal{F}_t$-measurable, as it is the almost-sure limit and the filtration is complete. This completes the proof.
	\end{proof}
	It remains to show that $u_1\neq u_2$. Let us consider the $L^r$-norm of $u_1-u_2$. Thanks to \cite[Theorem 3.3]{glogic2025non}, the regularity of $z$, \autoref{non_linear_singular_pde_noise} and relation \eqref{item_instability}, we get that for each $t\leq \mathcal{T}_1\wedge \mathcal{T}_2$
	\begin{align*}
		\norm{u_1(t)-u_2(t)}_{L^r(\R^d)}&\geq {t^{-\left(\frac{1}{p-1}-\frac{d}{2r}\right)}}\norm{\psi(\ln t)}_{L^r(\R^d)}-\norm{w_1-w_2}_{C([0,T'],L^r(\R^d))}\\ &\gtrsim t^{-\left(\frac{1}{p-1}-\frac{d}{2r}-\lambda_{\bar{\alpha}}\right)}\frac{\norm{\bar{U}^{lin}}_{L^{r}(\R^d)}}{2}-1\rightarrow +\infty \quad\mathbb{P}-a.s.
	\end{align*}
	as $t\rightarrow 0$.
\section{Almost Everywhere Local Well-Posedness for Random Initial Data}\label{Random_data_well_posed} 
\subsection{Randomization Estimates}
In this section we study the effect of the randomization \eqref{randomization} to the regularity on the initial conditions $u_0^{\omega}$ and the linear flow $e^{t\Delta}u_0^{\omega}$. This is addressed, respectively, in \autoref{reg_initial_condition} and \autoref{reg_stoch_conv} below.
\begin{lemma}\label{reg_initial_condition}
If $u_0\in M^{q,q}(\R^d)$ for $q<2$, then $u_0^\omega\in L^q(\R^d)$ $\mathbb{P}-a.s.$ If $u_0\in L^q(\R^d)\cap L^2(\R^d)$ for $q\geq 2$, then $u_0^\omega\in L^q(\R^d)\cap L^2(\R^d)$ $\mathbb{P}-a.s.$  
\end{lemma}
\begin{proof}
    In the case $q<2$, we have, by the embedding of $M^{q,q}(\R^d)\hookrightarrow L^q(\R^d)$ \emph{cf.} \cite[Theorem 1.4]{kobayashi2011inclusion} and Minkowski's inequality
    \begin{align*}
        \mathbb{E}\left[\norm{u_0^{\omega}}_{L^q}^2\right]&\lesssim_{q}\mathbb{E}\left[\norm{u_0^{\omega}}_{M^{q,q}}^2\right]\\ & =\mathbb{E}\left[\left(\sum_{k\in \Z^d}\norm{P_k u_0}^q_{L^q}\lvert h_k\rvert^q\right)^{2/q}\right]\\ & \lesssim_q \left(\sum_{k\in \Z^d} \left(\mathbb{E}\left[\norm{P_k u_0}_{L^q}^2\lvert h_k\rvert^2\right]\right)^{q/2}\right)^{2/q}\\ & =\norm{u_0}_{M^{q,q}}^2.
    \end{align*}
  In the case $q\geq 2 $ we have instead for each $\rho\geq 2$ by ultracontractivity of Gaussian random variables, see for example \cite[Lemma 2.1]{burq2008random} for a proof in this framework, Minkowski's integral inequality and \emph{unit-scale} Bernstein estimate, see for example \cite[Lemma 2.1]{luhrmann2014random}
  \begin{align*}
    \mathbb{E}\left[\norm{u_0^{\omega
      }}_{L^{\rho}}^{\rho}\right]& \lesssim {\rho}^{{\rho}/2} \norm{\left(\sum_{k\in \Z^d}\lvert P_k u_0\rvert^2\right)^{1/2}}_{L^{{\rho}}}^{{\rho}}\\ & \lesssim {\rho}^{{\rho}/2} \left(\sum_{k\in \Z^d}\norm{P_k u_0}_{L^\rho}^2\right)^{\rho/2}\\ & \lesssim {\rho}^{{\rho}/2} \left(\sum_{k\in \Z^d}\norm{P_k u_0}_{L^2}^2\right)^{\rho/2}\lesssim {\rho}^{{\rho}/2}  \norm{u_0}_{L^2}^{\rho}.
  \end{align*}
  The latter implies the claim by choosing $\rho=2$ or $\rho=q\geq 2.$
\end{proof}
\begin{remark}\label{further_integrability_rmk}
    As the proof above shows, if $u_0\in L^q(\R^d)\cap L^2(\R^d)$ for $q\geq 2$ then $u_0^\omega\in L^{\rho}(\R^d) \ \mathbb{P}-a.s.$ for each $\rho\in [2,+\infty).$
\end{remark}

\begin{lemma}\label{reg_stoch_conv}
    Let $T>0$ and $u_0\in H^{-\alpha}(\R^d)$ for some $ \alpha\geq 0$. Let $\theta_1\geq \theta_2\geq \theta_3\geq 2,\ \sigma>0$ and $\gamma\in \mathbb{R}$ such that \begin{align}\label{integrability_cond}
        (\sigma+\alpha-2\gamma)\theta_3<2.
    \end{align} 
    Then 
    \begin{align}\label{regularity_bound}
     \norm{t^\gamma (I-\Delta )^{\sigma/2}e^{\Delta t}u_0^{\omega}}_{L^{\theta_1}(\Omega;L^{\theta_3}(0,T;L^{\theta_2}(\R^d)))}& \leq C_T \sqrt{\theta_1}\norm{u_0}_{H^{-\alpha}},   
    \end{align}
    for some $C_T>0$ depending on $T,\alpha,\sigma,\gamma,\theta_2,\theta_3$ and approaching $0$ as $T\rightarrow 0$.
    Moreover, by calling
\begin{align*}
    E_{\lambda,T,u_0}=\{\omega\in \Omega:  \norm{t^\gamma (I-\Delta )^{\sigma/2}e^{t\Delta}u_0^{\omega}}_{L^{\theta_3}(0,T;L^{\theta_2}(\R^d))}\geq \lambda \},
\end{align*}
then there exist $c_1,c_2$ independent of $\lambda, u_0$ such that
\begin{align}\label{exponential_estimate}
    \mathbb{P}\left(E_{\lambda,T,u_0}\right)\leq c_1 e^{-c_2\frac{\lambda^2}{C_T\norm{u_0}^2_{H^{-\alpha}(\R^d)}}}.
\end{align}
\end{lemma}
\begin{proof}
Let us start by observing that $(I-\Delta )^{\sigma/2}e^{t\Delta }$ and the randomization commute.
\begin{align*}
    (I-\Delta )^{(\sigma+\alpha)/2}e^{t\Delta }u_0^{\omega}&=\left( (I-\Delta )^{(\sigma+\alpha)/2}e^{t\Delta }u_0\right)^{\omega}\\ & =\sum_{k\in \Z^d}\mathcal{F}^{-1}(\psi_k(\xi)(1+\lvert \xi\rvert^{2})^{(\sigma+\alpha)/2}e^{-\lvert \xi\rvert^2 t}\hat{u}_0(\xi)) h_k.
\end{align*}
Therefore, by Minkowski's integral inequality and ultracontractivity of Gaussian random variables, see e.g. \cite[Lemma 2.1]{burq2008random} for a proof in this framework, it holds by calling $u_k^{\alpha,\sigma}(x,t):=\mathcal{F}^{-1}\left(\psi_k(\xi)(1+\lvert \xi\rvert^{2})^{(\sigma+\alpha)/2}e^{-\lvert \xi\rvert^2 t}\frac{\hat{u}_0(\xi)}{(1+\lvert \xi\rvert^{2})^{\alpha/2}}\right)(x) ,$
\begin{align*}
 \norm{t^\gamma (I-\Delta )^{\sigma/2}e^{t\Delta }u_0^{\omega}}_{L^{\theta_1}(\Omega;L^{\theta_3}(0,T;L^{\theta_2}(\R^d)))} &= \norm{t^\gamma \sum_{k\in \Z^d} u_k^{\alpha,\sigma}(x,t) h_k(\omega)}_{L^{\theta_1}(\Omega;L^{\theta_3}(0,T;L^{\theta_2}(\R^d)))}   \\ & \lesssim \sqrt{\theta_1}\left(\int_0^T \left(t^{2\gamma}\sum_{k\in \Z^d} \norm{ u_k^{\alpha,\sigma}(x,t) }_{L^{\theta_2}(\R^d)}^2\right)^{\theta_3/2}\right)^{\frac{1}{\theta_3}}.
\end{align*}
Moreover, by \emph{unit-scale} Bernstein estimate, see for example \cite[Lemma 2.1]{luhrmann2014random}, since $\theta_2\geq 2$ we have also $\norm{ u_k^{\alpha,\sigma}(x,t) }_{L^{\theta_2}(\R^d)}\lesssim_{\theta_2} \norm{ u_k^{\alpha,\sigma}(x,t) }_{L^{2}(\R^d)}$. Therefore \begin{align*}
    \norm{t^\gamma (I-\Delta )^{\sigma/2}e^{t\Delta }u_0^{\omega}}_{L^{\theta_1}(\Omega;L^{\theta_3}(0,T;L^{\theta_2}(\R^d)))} & \lesssim_{\theta_2} \sqrt{\theta_1}\left(\int_0^T \left(t^{2\gamma}\sum_{k\in \Z^d} \norm{ u_k^{\alpha,\sigma}(x,t) }_{L^{2}(\R^d)}^2\right)^{\theta_3/2}\right)^{\frac{1}{\theta_3}}.
\end{align*}
The latter can be estimated directly. Indeed, using Parseval’s identity
\begin{align*}
    \sum_{k\in \Z^d} \norm{ u_k^{\alpha,\sigma}(x,t) }_{L^{2}(\R^d)}^2&= \sum_{k\in \Z^d}\int_{\R^d} \psi_k^2(\xi)(1+\lvert \xi\rvert^{2})^{\sigma+\alpha}e^{-2\lvert \xi\rvert^2 t}\frac{\lvert \hat{u}_0(\xi)\rvert^2}{(1+\lvert \xi\rvert^{2})^{\alpha}} d\xi\\ & \lesssim_{\sigma,\alpha} t^{-(\sigma+\alpha)}\sum_{k\in \Z^d}\int_{\R^d} \psi_k^2(\xi)\frac{\lvert \hat{u}_0(\xi)\rvert^2}{(1+\lvert \xi\rvert^{2})^{\alpha}} d\xi\\ & \lesssim_{\alpha} t^{-(\sigma+\alpha)}\norm{u_0}_{H^{-\alpha}(\R^d)}^2.
\end{align*}
Consequently
\begin{align*}
    \norm{t^\gamma (I-\Delta )^{\sigma/2}e^{t\Delta }u_0^{\omega}}_{L^{\theta_1}(\Omega;L^{\theta_3}(0,T;L^{\theta_2}(\R^d)))} & \lesssim_{\alpha,\sigma,\theta_2} \sqrt{\theta_1}\norm{u_0}_{H^{-\alpha}(\R^d)}\left(\int_0^T \left(t^{2\gamma-\sigma-\alpha}\right)^{\theta_3/2}\right)^{\frac{1}{\theta_3}}\\ & \lesssim_{\alpha,\sigma,\gamma,\theta_3}\sqrt{\theta_1}T^{\frac{1}{\theta_3}+\gamma-\frac{\sigma+\alpha}{2}}\norm{u_0}_{H^{-\alpha}(\R^d)}
\end{align*} 
under the assumption \eqref{integrability_cond}. This relation readily implies \eqref{regularity_bound}. Let us now move to \eqref{exponential_estimate}.
First let us observe that \eqref{exponential_estimate} trivially follows if $\lambda \leq \sqrt{\theta_2} \norm{u_0}_{H^{-\alpha}} e C_T$ for whatever choice of $c_2$ and $c_1\geq e^{c_2\theta_2 e^2 C_T^2}.$ In case of $\lambda >\sqrt{\theta_2} \norm{u_0}_{H^{-\alpha}}eC_T $, setting $\theta_1=\frac{\lambda^2}{\norm{u_0}_{H^{-\alpha}}^2 e^2C_T^2}\geq \theta_2$, by Markov's inequality
\begin{align*}
    \mathbb{P}\left(E_{\lambda,T,u_0}\right)&\leq \lambda^{-\theta_1}(C_T\sqrt{\theta_1}\norm{u_0}_{H^{-\alpha}})^{\theta_1}\\ & \leq e^{\frac{-\lambda^2}{\norm{u_0}_{H^{-\alpha}}^2 e^2 C_T^2}}.
\end{align*}
Therefore, the claim follows.
\end{proof}
\subsection{Proof of \autoref{Fixed_point_random_initial_cond}}\label{proof_thm_random}
Due to \autoref{reg_initial_condition}, we already know that ${u_0}^{\omega}\in L^q(\R^d)$ $\mathbb{P}$-a.s. Let us denote by $\hat{z}(t):=e^{t\Delta }u_0^{\omega}$.
Therefore ,in order to construct a mild solution $u$ of \eqref{Eq:Heat}, it is enough to find a solution in $C([0,\mathcal{T}];L^r(\R^d))$ of the integral equation
\begin{align*}
    v(t)=\int_0^t e^{(t-s)\Delta }\lvert v(s)+\hat{z}(s)\rvert^{p-1}(v(s)+\hat{z}(s))ds
\end{align*}
for some stopping time $\mathcal{T}.$ Then a posteriori we have to recover $v\in C([0,\mathcal{T}],L^q(\R^d))$. Let us recall that 
\begin{align*}
    r=q_c(p\vee p')>q_c,
\end{align*}
Therefore $C([0,\mathcal{T}];L^r(\R^d))$ is a subcritical space for \eqref{Eq:Heat}.
We divide the proof into two steps according to the plan described above.\\
\emph{Step 1: Fixed Point.}
Let us observe that in case of $q\geq 2,\ u_0\in L^2(\R^d)$, while, in case of $q<2$, under our assumptions $u_0\in H^{-s_0}(\R^d)$ with 
\begin{align*}
    s_0=\frac{d(2-q)}{2q}<\frac{2}{p}
\end{align*}
by Sobolev embedding. Therefore in both cases we can apply \autoref{reg_stoch_conv} to properly estimate $\hat{z}$ when needed. Moreover, due to \autoref{reg_initial_condition} and standard regularizing properties of the heat semigroup, we have $\hat{z}\in C_{loc}([0,+\infty);L^q(\R^d))\cap C_{loc}([\varepsilon,+\infty);W^{s,q}(\R^d))$ $ \mathbb{P}-a.s.$ for each $s,\varepsilon>0.$ We want to construct a fixed point of the map 
\begin{align}\label{map_fixed_point}
  \Gamma[v](t):=\int_0^t e^{(t-s)\Delta }\lvert v(s)+\hat{z}(s)\rvert^{p-1}(v(s)+\hat{z}(s))ds  
\end{align}
in the closed ball of radius $M$ of $C([0,\hat{T}];L^r(\R^d))$ for sufficiently small $M$ and $\hat{T}\leq 1.$ We shall denote this closed ball by $B_{M,\hat{T}}$ in the following and split the proof in two sub-cases if $q\geq 2$ or not.\\
\emph{Step 1, Case 1: $ q\geq 2$.}
First we show that $\Gamma$ maps $B_{M,\hat{T}}$ in itself. By ultracontractivity of the heat semigroup and Young's inequality we have
\begin{align*}
    \norm{\Gamma[v](t)}_{L^r}&\leq C_{d,p,r}\int_0^t \left(\frac{1}{(t-s)^{\frac{d(p-1)}{2r}}}\norm{v(s)}_{L^r}^p+\norm{\hat{z}(s)}_{L^{rp}}^p\right)ds\\ & \leq C_{d,p,r} \left(M^p \hat{T}^{1-\frac{d(p-1)}{2r}}+\hat{T}^{1/2}\norm{\hat{z}}_{L^{2p}(0,\hat{T};L^{rp})}^p\right),
\end{align*}
where $C_{d,p,r}$ is a finite constant, possibly changing value line by line. In the last step we used the definition $B_{M,\hat{T}},$ the fact that $r>q_c$ and \autoref{reg_stoch_conv} for $\alpha=\sigma=\gamma=0.$ Therefore
$\norm{\Gamma[v](t)}_{L^r}\leq M$
if $M,\ \hat{T}$ are sufficiently small, for example 
\begin{align*}
    M=\frac{1}{(2C_{d,p,r})^{1/(p-1)}},\ \hat{T}=\operatorname{inf}\left\{t>0: C_{d,p,r} t^{1/2}\norm{\hat{z}}_{L^{2p}(0,t;L^{rp})}^p\geq \frac{M}{2}\right\}.
\end{align*}
The latter is a stopping time due to the $\mathcal{F}_0$ measurability of $u_0^{\omega}$. The continuity of $\Gamma[v](t)$ follows similarly.
Now we are left to show that $\Gamma$ is a contraction. We employ the elementary estimate
\begin{align}\label{elementary_estimate}
    \left\lvert\lvert x+y\rvert^{p-1}(x+y)-\lvert x+z\rvert^{p-1}(x+z)\right\rvert&\lesssim_p \left(\lvert x\rvert^{p-1}+\lvert y\rvert^{p-1}+\lvert z\rvert^{p-1}\right)\lvert y-z\rvert
\end{align}
in the following. Let us analyze $\Gamma[v_1]-\Gamma[v_2]$. Inequality \eqref{elementary_estimate} together with ultracontractivity of the heat semigroup and H\"older inequality imply
\begin{align*}
    \norm{\Gamma[v_1](t)-\Gamma[v_2](t)}_{L^r}& \leq C_{d,p,r}\int_0^t \frac{\norm{\left(\lvert \hat{z}(s)\rvert^{p-1}+\lvert v_1(s)\rvert^{p-1}+\lvert v_2(s)\rvert^{p-1}\right)\lvert v_1(s)-v_2(s)\rvert}_{L^{r/p}}}{(t-s)^{\frac{d(p-1)}{2r}}} ds\\ & \leq C_{d,p,r} \left(\hat{T}^{1-\frac{d(p-1)}{2r}}M^{p-1}+\hat{T}^{1-\frac{d(p-1)}{2(r-p+1)}}\norm{\hat{z}}_{L^r(0,\hat{T};L^r)}^{p-1}\right)\norm{v_1-v_2}_{C([0,\hat{T}];L^r)}.
\end{align*}
Therefore
$\Gamma[v]$ is a contraction on $B_{M,\hat{T}}$ if $M,\hat{T}$ are sufficiently small, for example 
\begin{align*}
    {M}&=\frac{1}{(2C_{d,p,r})^{1/(p-1)}},\\ \hat{T}&=\operatorname{inf}\left\{t>0: C_{d,p,r} t^{1-\frac{d(p-1)}{2(r-p+1)}} \norm{z}_{L^{r}(0,t;L^{r})}^{p-1}\geq \frac{1}{4}\right\}\wedge \operatorname{inf}\left\{t>0: C_{d,p,r}t^{1/2} \norm{z}_{L^{2p}(0,t;L^{rp})}^p\geq \frac{M}{2}\right\}.
\end{align*}
Again, the latter is a stopping time due to the $\mathcal{F}_0$ measurability of $u_0^{\omega}$. Let us call $\mathcal{T}=\hat{T}$, this is the stopping time in the statement of \autoref{Fixed_point_random_initial_cond}. Lastly, by the last part of \autoref{reg_stoch_conv}, up to renaming the constants appearing there, there exist two constants $C^1_{d,p,r}, C^2_{d,p,r}$ such that
\begin{align*}
    \mathbb{P}(\mathcal{T}\geq T)\geq 1- C^1_{d,p,r} e^{-\frac{C^2_{d,p,r}}{T^{\frac{1}{p}\wedge\left(\frac{2}{p-1}-\frac{d}{r-p+1}\right)}\norm{u_0}_{L^2}^2}}=1- C^1_{d,p,r} e^{-\frac{C^2_{d,p,r}}{T^{\frac{1}{p}}\norm{u_0}_{L^2}^2}},
\end{align*}
since by elementary computations one can check that 
\begin{align*}
  \frac{1}{p} \leq  \frac{2}{p-1}-\frac{d}{r-p+1}
\end{align*}
under the assumption $r=q_c(p\vee p').$\\
\emph{Step 1, Case 2: $2-\frac{8}{4+dp} <q< 2$.}
First we show that $\Gamma$ maps $B_{M,\hat{T}}$ into itself. By ultracontractivity of the heat semigroup, the definition $B_{M,\hat{T}} $ and H\"older's inequality we have
\begin{align*}
    \norm{\Gamma[v](t)}_{L^r}&\leq C_{d,p,r}\int_0^t \left(\frac{1}{(t-s)^{\frac{d(p-1)}{2r}}}\norm{v(s)}_{L^r}^p+\frac{\norm{ s^{\gamma}\hat{z}(s)}_{L^{rp}}^p}{s^{\gamma p}}\right)ds\\ & \leq C_{d,p,r,\gamma,\alpha} \left(M^p \hat{T}^{1-\frac{d(p-1)}{2r}}+\hat{T}^{\frac{\alpha-1}{\alpha}\left(1-\gamma p\frac{\alpha}{\alpha-1}\right)}\norm{t^{\gamma}\hat{z}(t)}_{L^{\alpha p}(0,\hat{T};L^{rp})}^p\right),
\end{align*}
where $C_{d,p,r,\gamma,\alpha}$ is a finite constant if it is possible to find $\alpha\in [1,+\infty],\ \gamma\in \R$ such that
\begin{align*}
    \begin{cases}
        r\geq \alpha\geq 2,\\
        (s_0-2\gamma)\alpha p<2,\\
        \gamma p\frac{\alpha}{\alpha-1}<1.
    \end{cases}
\end{align*} 
By setting 
\begin{align*}
    \gamma=\frac{\alpha-1}{p\alpha}-\frac{\epsilon}{2\alpha p}.
\end{align*}
for $\epsilon>0$ the third condition is immediately satisfied and we are left to find $\alpha$ and $\epsilon$ such that
\begin{align*}
    \begin{cases}
        r\geq \alpha\geq 2,\\
       \alpha(-2+s_0 p)<-\epsilon.
    \end{cases}
\end{align*} 
Since under our assumptions on $q$ we have $(s_0p -2)<0$, all the conditions above are satisfied choosing 
\begin{align*}
    \alpha=2,\quad \epsilon=2(2-s_0 p)\left(1-\frac{\epsilon' }{2-s_0p}\right),\quad \gamma=\frac{s_0 p-1+\epsilon'}{2p},\quad 0<\epsilon'<\frac{\alpha(2-s_0p)}{2}.
\end{align*}
In conclusion we showed 
\begin{align*}
    \norm{\Gamma[v](t)}_{L^r}& \leq C_{d,p,r,\epsilon'} \left(M^p \hat{T}^{1-\frac{d(p-1)}{2r}}+\hat{T}^{\frac{2-s_0p }{4}}\norm{t^{\frac{s_0 p-1+\epsilon'}{2p}}\hat{z}(t)}_{L^{2 p}(0,\hat{T};L^{rp})}^p\right). 
\end{align*}
In particular, due to \autoref{reg_stoch_conv}, the quantity above is always finite and
$\norm{\Gamma[v](t)}_{L^r}\leq M$
by choosing $M,\ \hat{T}$ sufficiently small, for example 
\begin{align*}
    M=\frac{1}{(2C_{d,p,r,\epsilon'})^{1/(p-1)}},\ \hat{T}=\operatorname{inf}\left\{t>0: C_{d,p,r,\epsilon'} t^{\frac{2-s_0p }{4}}\norm{s^{\frac{s_0 p-1+\epsilon'}{2p}}\hat{z}}_{L^{2 p}(0,t;L^{rp})}^p\geq \frac{M}{2}\right\}.
\end{align*}
The latter is a stopping time due to the $\mathcal{F}_0$ measurability of $u_0^{\omega}$. The continuity of $\Gamma[v](t)$ follows similarly.
Now we are left to show that $\Gamma$ is a contraction.
Let us analyze $\Gamma[v_1]-\Gamma[v_2]$. By \eqref{elementary_estimate}, ultracontractivity of the heat semigroup and H\"older inequality we have, arguing similarly as above,
\begin{align}\label{contraction_estimate}
   & \norm{\Gamma[v_1](t)-\Gamma[v_2](t)}_{L^r} \leq C_{d,p,r}\int_0^t \frac{\norm{\left(\lvert \hat{z}(s)\rvert^{p-1}+\lvert v_1(s)\rvert^{p-1}+\lvert v_2(s)\rvert^{p-1}\right)\lvert v_1(s)-v_2(s)\rvert}_{L^{r/p}}}{(t-s)^{\frac{d(p-1)}{2r}}} ds\\  &\leq C_{d,p,r} \left(\hat{T}^{1-\frac{d(p-1)}{2r}}M^{p-1}+\left(\int_0^{\hat{t}} \frac{1}{s^{\frac{\alpha'}{\alpha'-1}\gamma'(p-1)}(t-s)^{\frac{d(p-1)\alpha}{2r(\alpha-1)}}}ds\right)^{\frac{\alpha-1}{\alpha}}\norm{t^{\gamma'}\hat{z}}_{L^{\alpha'(p-1)}(0,\hat{T};L^r)}^{p-1}\right)\notag\\ & \quad \quad\quad\quad\times\norm{v_1-v_2}_{C([0,\hat{T}];L^r)}\notag.
\end{align}
According to \autoref{reg_stoch_conv}, the latter is a treatable quantity, finite $\mathbb{P}$-almost surely, as soon as we can find $\alpha'$ and $\gamma'$ such that
\begin{align*}
    \begin{cases}
        r\geq \alpha'(p-1)\geq 2\\
        \frac{d(p-1)\alpha'}{2r(\alpha'-1)}<1\\
        \frac{\alpha'}{\alpha'-1}\gamma'(p-1)<1\\
        (s_0-2\gamma')\alpha'(p-1)<2.
    \end{cases}
\end{align*}
Considerations analogous to the ones above allow us to find the choice of parameters 
\begin{align*}
    \alpha'=p'\vee \frac{2}{p-1},\quad \gamma'=\frac{\alpha'-1}{\alpha'(p-1)}-\frac{\alpha'(2-s_0(p-1))-\epsilon'}{2\alpha(p-1)},\quad 0<\epsilon'<\alpha'(2-s_0(p-1)),
\end{align*}
where $p'$ is the conjugate of $p$.
In particular, under our choice of the parameters
\begin{align*}
   \int_0^t \frac{1}{s^{\frac{\alpha'}{\alpha'-1}\gamma'(p-1)}(t-s)^{\frac{d(p-1)\alpha}{2r(\alpha-1)}}}ds & \lesssim \left(t^{\frac{\alpha'}{2(\alpha'-1)}}\right)^{2-s_0(p-1)-\epsilon'\alpha'-\frac{2 }{p\vee p'}}.
\end{align*}
By elementary computations one can check that 
\begin{align*}
    2-s_0(p-1)-\frac{2 }{p\vee p'}> 2-\frac{2}{p}(p-1)-\frac{2 }{p\vee p'}\geq 0.
\end{align*}
Therefore, if we choose
\begin{align}\label{condition_small_epsilon}
    0<\epsilon'<\frac{2-s_0(p-1)-\frac{2 }{p\vee p'}}{2\alpha'}\wedge \alpha'(2-s_0(p-1))\wedge \frac{\alpha(2-s_0p)}{2}
\end{align}
estimate \eqref{contraction_estimate} reduces to
\begin{align*}
&\norm{\Gamma[v_1](t)-\Gamma[v_2](t)}_{L^r} \\ & \leq 
   C_{d,p,r,q} \left(\hat{T}^{1-\frac{d(p-1)}{2r}}M^{p-1}+\hat{T}^{\frac{2-s_0(p-1)-\frac{2 }{p\vee p'}}{4}}\norm{t^{\gamma'}\hat{z}}_{L^{\alpha'(p-1)}(0,\hat{T};L^r)}^{p-1}\right)\norm{v_1-v_2}_{C([0,\hat{T}];L^r)}. 
\end{align*}
Therefore
$\Gamma[v]$ is a contraction on $B_{M,\hat{T}}$ if we choose $M,\hat{T}$ sufficiently small, for example 
\begin{align*}
    {M}&=\frac{1}{(2C_{d,p,r})^{1/(p-1)}},\\ \hat{T}&=\operatorname{inf}\left\{t>0: C_{d,p,q,r} 
    t^{\frac{2-s_0(p-1)-\frac{2 }{p\vee p'}}{4}}\norm{s^{\gamma'}\hat{z}}_{L^{\alpha'(p-1)}(0,t;L^r)}^{p-1}\geq \frac{1}{4}\right\}\\ & \wedge \operatorname{inf}\left\{t>0: C_{d,p,r,\epsilon'} t^{\frac{2-s_0p }{4}}\norm{s^{\frac{s_0 p-1+\epsilon'}{2p}}\hat{z}}_{L^{2 p}(0,t;L^{rp})}^p\geq \frac{M}{2}\right\},
\end{align*}
the latter being a stopping time due to the $\mathcal{F}_0$ measurability of $u_0^{\omega}$. Let us call $\mathcal{T}=\hat{T}$. This is the stopping time in the statement of \autoref{Fixed_point_random_initial_cond}. Moreover, by the last part of \autoref{reg_stoch_conv}, up to renaming the constant appearing there, there exist two constants $C^1_{d,p,q,r}, C^2_{d,p,q,r}$ such that
\begin{align*}
    \mathbb{P}(\mathcal{T}\geq T)\geq 1- C^1_{d,p,r} e^{-\frac{C^2_{d,p,r}}{T^{\frac{1}{p}\wedge\left(\frac{2}{p-1}-\frac{d}{r-p+1}\right)}\norm{u_0}_{H^{-s_0}}^2}}\geq 1- C^1_{d,p,r} e^{-\frac{C^2_{d,p,q,r}}{T^{\frac{1}{p}\wedge\left(\frac{2}{p-1}-\frac{d}{r-p+1}\right)}\norm{u_0}_{M^{q,q}}^2}}.
\end{align*}

\emph{Step 2: $u\in C([0,\mathcal{T}],L^q(\R^d))$ and it is an adapted process.} A priori, since we are working on $\R^d$, the regularity of $v$ does not imply that $u\in C([0,\mathcal{T}],L^q(\R^d))$. However, arguing similarly to \autoref{Continuity_uniform_bound} and exploiting the regularity of $e^{t\Delta }u_0^{\omega}$ this can be proved. Since the proof of this claim is pretty similar to \autoref{Continuity_uniform_bound}, we only sketch the details. We note that $\hat{z}(t):=e^{t\Delta }u_0^{\omega}\in C_{loc}([0,+\infty),L^q(\R^d))\quad \mathbb{P}-a.s.$ and it is adapted. Therefore it is enough to show that the claim holds for $v=u-\hat{z}.$ Let us introduce a radially symmetric cut off $\chi\in C^{\infty}_c(\R^d)$ by
		\begin{align*}
			\chi(x):=\begin{cases}
				1\quad \text{if } \lvert x\rvert\leq 1\\
				0\quad \text{if } \rvert x\rvert\geq 2,
			\end{cases}\quad \chi(x)\geq 0,\quad \chi(x)\leq \chi(x') \quad \text{if} \quad \lvert x\rvert\geq \lvert x'\rvert.
		\end{align*}
		Furthermore, for $x_0\in \R^d,\, R\geq 1$, define $\chi_{x_0, R}(x):=\chi\left(\frac{x-x_0}{R}\right).$ Due to the properties of $v$, we have that for each $x_0\in \R^d,\, R\geq 1$ \begin{align}\label{regularity_cutoff_random_IC}
			v\chi_{x_0,R}\in C([0,\mathcal{T}],L^q(\R^d))\cap C([0,\mathcal{T}],L^r(\R^d))\quad \mathbb{P}-a.s.,
		\end{align}            
		and for each $t,t_0\in (0,\mathcal{T})$ with $t_0\leq t$, that
		\begin{align*}
			v(t)\chi_{x_0,R}= \, &e^{(t-t_0)\Delta }v(t_0)\chi_{x_0,R}+\int_{t_0}^t e^{(t-s)\Delta }\left(v(s)\Delta\chi_{x_0,R}-2\operatorname{div}(v(s)\nabla \chi_{x_0,R})\right) ds\\ & +\int_{t_0}^t e^{(t-s)\Delta } \chi_{x_0,R} \lvert v(s)+\hat{z}(s)\rvert^{p-1}(v(s)+\hat{z}(s)) ds,
		\end{align*}
		in  $L^q(\R^d)$.
		Considering the $L^q$ norm of the equation above, by triangle inequality and the definition of $\chi_{x_0,R}$ we get easily  that
		\begin{align*}
			\norm{ v(t)}_{L^q(B_{R}(x_0))}\lesssim \, & \norm{v(t_0)\chi_{x_0,R}}_{L^q(\R^d)}+\int_{t_0}^t \left(1+\frac{1}{\sqrt{t-s}}\right) \norm{ v(s)}_{L^q(B_{2R}(x_0))}ds\\ & +\int_{t_0}^t\norm{ \chi_{x_0,R} \lvert v(s)+\hat{z}(s)\rvert^{p-1}(v(s)+\hat{z}(s))}_{L^q(\R^d)} ds .
		\end{align*}
		Here the hidden constant is independent of $R$, $t_0,$ $t$ and $x_0$. Let us analyze further the last term.
		By interpolation and Young's inequality, since $r>pq$, we have
		\begin{align*}
			\norm{ \chi_{x_0,R} \lvert v(s)+\hat{z}(s)\rvert^{p-1}(v(s)+\hat{z}(s))}_{L^q(\R^d)}\lesssim \, & \norm{ \chi^{1/p}_{x_0,R} v(s)}_{L^{pq}(\R^d)}^p+\norm{\hat{z}(s)}_{L^{q,r}(\R^d)}^p\\ \leq \, & \norm{ \chi^{1/p}_{x_0,R} v(s)}_{L^{q}(\R^d)}^{\frac{r-pq}{r-q}}\norm{  v(s)}_{L^{r}(\R^d)}^{\frac{r(p-1)}{r-q}}+\norm{\hat{z}(s)}_{L^{q,r}(\R^d)}^p\\ \lesssim \, & \norm{  v(s)}_{L^{q}(B_{2R}(x_0))} +\norm{v}_{C([0,\mathcal{T}],L^r(\R^d))}^{\frac{r}{q}}+\norm{\hat{z}(s)}_{L^{q,r}(\R^d)}^p.
		\end{align*}
		Therefore by \autoref{reg_stoch_conv}, arguing similarly to \emph{Step 1} above,  
		\begin{align*}
			\norm{ v(t)}_{L^q(B_{R}(x_0))}\lesssim \,& 1+\norm{v(t_0)\chi_{x_0,R}}_{L^q(\R^d)}+\int_{t_0}^t \left(1+\frac{1}{\sqrt{t-s}}\right) \norm{ v(s)}_{L^q(B_{2R}(x_0))}ds.
		\end{align*}
        Possibly the hidden constant above depends on $\omega,$ but this is inessential.
		Due to relation \eqref{regularity_cutoff_random_IC}, since $v(0)=0$, we can let $t_0\rightarrow 0$ to get
		\begin{align*}
			\norm{ v(t)}_{L^q(B_{R}(x_0))}&\lesssim 1+\int_{0}^t \left(1+\frac{1}{\sqrt{t-s}}\right) \norm{ v(s)}_{L^q(B_{2R}(x_0))}ds.   
		\end{align*}
        Therefore $v\in L^{\infty}(0,\mathcal{T};L^q(\R^d))\quad \mathbb{P}-a.s.$ arguing verbatim as in the proof of \autoref{Continuity_uniform_bound}. In order to show the continuity of $v(t)$ in $L^q$, we use again the mild formulation for $v(t)$. Namely, for each $0\leq t_1\leq t_2\leq \mathcal{T}$, we have 
		\begin{align*}
			v(t_2)\chi_{x_0,R}-v(t_1)&=\left(e^{(t_2-t_1)\Delta }v(t_1)\chi_{x_0,R}-v(t_1)\right)\\ &+\int_{t_1}^{t_2} e^{(t-s)\Delta }\left(v(s)\Delta\chi_{x_0,R}-2\operatorname{div}(v(s)\nabla \chi_{x_0,R})\right) ds\\ & +\int_{t_1}^{t_2} e^{(t-s)\Delta } \chi_{x_0,R} \lvert v(s)+\hat{z}(s)\rvert^{p-1}(v(s)+\hat{z}(s)) ds,
		\end{align*}
		in $L^q(\R^d)$.   
		Taking the $L^q$-norm of the expression above and letting $R\rightarrow +\infty$, we obtain, by analogous considerations to the ones employed to obtain the uniform bound on the $L^q$-norm
		\begin{align*}
			\norm{v(t_2)-v(t_1)}_{L^q(\R^d)} = \, &\limsup_{R\rightarrow +\infty}\norm{v(t_2)\chi_{x_0,R}-v(t_1)}_{L^q(\R^d)}\\  \leq \, & \norm{e^{(t_2-t_1)\Delta }v(t_1)-v(t_1)}_{L^q(\R^d)}\\ & +\sup_{R\geq 1}\norm{\int_{t_1}^{t_2} e^{(t-s)\Delta }\left(v(s)\Delta\chi_{x_0,R}-2\operatorname{div}(v(s)\nabla \chi_{x_0,R})\right)ds}_{L^q(\R^d)}\\ & +\sup_{R\geq 1}\norm{\int_{t_1}^{t_2} e^{(t-s)\Delta } \chi_{x_0,R} \lvert v(s)+\hat{z}(s)\rvert^{p-1}(v(s)+\hat{z}(s)) ds }_{L^q(\R^d)}\\  \lesssim\, & \norm{e^{(t_2-t_1)\Delta }v(t_1)-v(t_1)}_{L^q(\R^d)}\\ & +\int_{t_1}^{t_2} \left(1+\frac{1}{\sqrt{t_2-s}}\right) \norm{ v(s)}_{L^q(\R^d)}ds+\int_{t_1}^{t_2} \norm{z(s)}_{L^{pq}}^p ds.
		\end{align*}
		Since we already proved that $v\in L^{\infty}(0,\mathcal{T},L^q(\R^d))\quad \mathbb{P}-a.s.,$ the continuity of $v$ then follows from the last inequality by \autoref{reg_stoch_conv} and arguing as in \emph{Step 1} above. \\
        Regarding measurability, the result follows exactly as in the last part of \autoref{Continuity_uniform_bound}.

\subsection*{Acknowledgements}
    The author is grateful to Martina Hofmanov\'a for valuable discussions on the topic and to Antonio Agresti for pointing out the reference \cite{van2008stochastic_lec}. The author also thanks Umberto Pappalettera for fruitful discussions and helpful comments on a preliminary version of the paper, which improved the presentation of the results.  
    The author has received funding from the European Research Council (ERC) under the European Union’s Horizon 2020 research and innovation programme (grant agreement No. 949981). 
	\bibliography{refs.bib}
	\bibliographystyle{plain}
\end{document}